\documentclass{article}
\usepackage[utf8]{inputenc}

\usepackage{amssymb}
\usepackage{amsmath}
\usepackage{amsthm}
\usepackage{enumitem}

\usepackage{wrapfig}
\usepackage{graphicx}
\usepackage{subcaption}
\usepackage{xcolor}
\usepackage{xparse}
\usepackage{ifthen}
\usepackage{enumitem}
\usepackage[title]{appendix}

\usepackage[hidelinks]{hyperref}

\title{On the maximum size packings of disks with kissing radius $3$}
\author{Alexander Golovanov\thanks{MIPT, Moscow, Russia. Email: \texttt{Golovanov@phystech.edu}}}
\date{}

\DeclareMathOperator{\parent}{Par}
\DeclareMathOperator{\proj}{proj}
\DeclareMathOperator{\angccw}{\measuredangle}

\renewcommand{\dj}{\delta}
\newcommand{\Dj}{\Delta}
\renewcommand{\Im}{\mathrm{Im}\,}
\newcommand{\ccwarc}[3]{\omega_{#1}(#2, #3)}

\makeatletter
\def\myhence#1{\expandafter\@myhence\csname c@#1\endcsname}
\def\@myhence#1{%
  \ifcase#1\or so\or then\or hence\or therefore\or thus\else\@ctrerr\fi}
\makeatother
\newcounter{hencectr}
\renewcommand{\Rightarrow}{%
\stepcounter{hencectr}%
\text{ and \myhence{hencectr}}\,%
\ifnum\the\value{hencectr}=5\setcounter{hencectr}{0}\fi%
}

\newtheorem{theorem}{Theorem}
\newtheorem{lemma}{Lemma}[section]
\newtheorem{claim}{Claim}[section]
\newtheorem{observation}[claim]{Observation}
\newtheorem{corollary}[claim]{Corollary}

\makeatletter
\newcommand{\optionaldesc}[3]{%
  #1\protected@edef\@currentlabel{#2}\ifthenelse{\equal{#3}{}}{}{\label{#3}}%
}
\makeatother

\newcommand\caseitem[3][]{\item[\optionaldesc{\textbf{Case #2: #3}}{#2}{\IfValueTF{#1}{#1}{}}]\mbox\\}

\theoremstyle{remark}
\newtheorem*{remark}{Remark}

\begin{document}

\maketitle

\begin{abstract}
    L\'{a}szl\'{o} Fejes T\'{o}th and Alad\'{a}r Heppes proposed the following generalization of the kissing number problem. Given a ball in $\mathbb{R}^d$, consider a family of balls touching it, and another family of balls touching the first family. Find the maximal possible number of balls in this arrangement, provided that no two balls intersect by interiors, and all balls are congruent. They showed that the answer for disks on the plane is $19$. They also conjectured that if there are three families of disks instead of two, the answer is $37$. In this paper we confirm this conjecture.
\end{abstract}

\section{Introduction}

A collection $\mathcal C$ of convex bodies in $\mathbb R^d$ is called a \emph{packing} if no two of them have an interior point in common. We assume that two convex bodies of $\mathcal{C}$ may touch each other. 
In the current paper, we are only interested in finite packings of disks (and balls) of unit diameters, and therefore, we omit the phrase ``of unit diameter'' most of the time.

The concept of \emph{minimum-distance graphs} is tightly connected with packings of balls.
Given a finite set of points in $\mathbb{R}^d$, the minimum-distance graph is the graph whose vertices are points of the set and edges are drawn between all pairs of points within the minimum distance among all pairwise distances in the set. In particular, if a packing of congruent balls in $\mathbb{R}^d$ contains at least one pair of touching balls, then there is a correspondence between pairs of touching balls and edges in the minimum-distance graph induced by the centeres of the balls. 
We refer the interested reader to the recent survey~\cite{Swanepoel2018} on various distance problems in discrete geometry, including those on packings.


Another concept closely related to the local structure of a packing of balls and minimum-distance graphs is the \emph{kissing number}. Recall that the kissing number in $\mathbb{R}^d$ is the maximum number of $d$-dimensional congruent balls with non-intersecting interiors that can touch a given ball of the same size. One can easily see that the kissing number of a space is the maximum possible degree of a vertex in a minimum-distance graph of any set of points from that space.
It is clear that the kissing number in the plane is $6$. One of the classical questions in mathematics is the thirteen spheres problem raised by Isaac Newton and David Gregory whether the kissing number in $\mathbb{R}^3$ is $12$ or $13$, which was first settled by K.~Sch\"{u}tte and B.\,L.~Van~der~Waerden in~\cite{schutte1952problem}. 

L. Fejes T\'{o}th and A. Heppes~\cite{toth_heppes} considered a generalization of the kissing number problem. Given a ball and a family of balls of the same size, all touching the initial one, add another family of balls of the same size that all of them touch at least one previous ball. No two balls may intersect by interiors. Find the maximal number $T_d$ of balls, except the first one, in such an arrangement. Fejes T\'{o}th and Heppes proved some lower and upper bounds for $T_2$, $T_3$, and $T_4$; for instance, $T_2 = 18$. In the last paragraph of the introduction of \cite{toth_heppes}, they also wrote:
\begin{quote}
We can continue the process of
successively enlarging a bunch of balls, but in the next step the problem becomes extremely intricate, even in the plane.
\end{quote}
Z.~F\"{u}redi and P.\,A.~Loeb~\cite[last paragraph of Section~6]{furedi} attribute the following conjecture to L.\,Fejes~T\'{o}th and A.~Heppes: The similar three-layer configuration in the plane contains at most $36$ disks (except the first given disk). The goal of this paper is to confirm this conjecture.

Let us introduce this problem formally. Given a packing $\mathcal{P}$ of disks, for two disks $A$, $B\in\mathcal{P}$ we define the \emph{kissing distance} between $A$ and $B$, as the smallest $n$ such that there is a sequence of disks $D_0$, $D_1$, \ldots, $D_n\in\mathcal{P}$, where $A = D_0$, $B = D_n$, and for $1\leq i\leq n$ the disks $D_{i-1}$ and $D_i$ touch each other. If there is no such sequence, then we set the kissing distance between $A$ and $B$ to $\infty$. We say that a packing $\mathcal{P}$ \emph{has kissing radius $n$} if there is a disk $D\in\mathcal{P}$, which we call \emph{the source} or \emph{the 0-disk}, such that for every other disk $B$ the kissing distance between $D$ and $B$ does not exceed $n$. Denote by $f(n)$ the maximum number of disks in a packing of kissing radius $n$.


In terms of minimum-distance graphs, $f(n)$ is the maximum possible number of vertices in the induced subgraph with just the vertices at distance $n$ or less from a given vertex. 

The hexagonal arrangement of disks gives the trivial lower bound $f(n) \geq 1 + {3n(n+1)}.$ Since the kissing number of the plane is $6$, we have $f(1) = 7$. According to~\cite{toth_heppes}, we have $f(2) = 19$. The main result of this paper is to confirm the conjecture $f(3) = 37$ mentioned in~\cite{furedi}.

\begin{theorem}\label{theorem:f_of_3} $f(3) = 37$.
\end{theorem}




It is also worth mentioning that this pattern does not generalize on all $n$. In fact, $f(n) = 2\pi/\sqrt3\cdot n^2\cdot(1 + o(1))$, but the paper with this result is in progress at the moment of writing.

This paper is organized as follows. In Section~2 we introduce the required notation and two key lemmas, Lemma~\ref{lemma:master} and Lemma~\ref{lemma:good_curve}. After this we deduce Theorem~\ref{theorem:f_of_3} from these lemmas. In Section~3 we prove Lemma~\ref{lemma:master}. The proof of Lemma~\ref{lemma:good_curve} is separated into constructing a curve of certain type (Section~4 and Appendix~A), reducing Lemma~\ref{lemma:good_curve} to a certain local inequality Claim~\ref{lemma:can-construct-without-alpha} (Section~5), and, finally, establishing this inequality (Section~6).


\medskip \textbf{Acknowledgements.}
The author thanks Alexey Balitskiy, Dmitry Krachun and Alexandr Polyanskii for helpful discussions and comments on drafts of this paper. The author also thanks Alexandr Polyanskii for simplifying the proof of Case~\ref{subsec:case-000}.

\section{Proof of Theorem \ref{theorem:f_of_3} modulo the key lemmas}



\subsection{Definitions}

Let $\mathcal{P}$ be a packing of kissing radius $n$ with $D$ being its source. A disk $D' \in \mathcal{P}$ is called a \textit{$d$-disk} if the kissing distance between it and $D$ equals $d$. Note that this definition justifies the name ``0-disk''. We call the set of $d$-disks the \textit{$d\textsuperscript{th}$ layer}. 

Consider any $k$-disk $A$, where $k > 0$. There is at least one disk $B$ from the $(k-1)$\textsuperscript{th} layer such that $A$ and $B$ touch each other. Pick any of these disks and call it the \textit{parent} of the disk $A$.
Denote the parent of the disk $A$ by $\parent(A)$.
In particular, $C_0$ is the parent of every $1$-disk. We will also say that each disk is a \emph{child} of its parent. A disk without children is called \emph{childfree}.

The crucial role in the argument will be played by curves of special type, which we now introduce. A curve $\gamma = \{\gamma(t)\,\colon\,t\in[0, \ell]\} \subset \mathbb{R}^2$ is \textit{sparse-centered} if the following conditions hold.

\begin{enumerate}[label=(\alph*)]
  \item There exists a sequence $t_0 = 0 < t_1 < \ldots < t_m=\ell$ such that $\gamma([t_{i-1}, t_i])$ is a circle arc of unit radius with center $c_i$ for all $i\in\{1, \ldots, m\}$;
  \item $\gamma(t)$ is a natural parametrization of $\gamma$, that is, $|\dot{\gamma}(t)| = 1$ for $t\in[0, \ell]\setminus\{t_i\}_{i=0}^m$, and thus, $|\gamma| = \ell$;
  \item For all $i\in\{1, \ldots, m\}$ and $t\in(t_{i-1}, t_i)$, the cross product $(\gamma(t) - c_i)\times \dot{\gamma}(t)$ is positive, that is, each arc is directed counterclockwise;
  \item For all $i$, $j\in\{1, \ldots, m\}$, either $c_i = c_j$ or $|c_j - c_i|\geq 1$.
\end{enumerate}

In the sequel we often omit the word ``sparse-centered'' when talking about curves.

\subsection{Outline of the proof of Theorem \ref{theorem:f_of_3}}

Consider an arbitrary packing of unit diameter disks of kissing radius $3$. Remove all $3$-disks and childfree $1$-disks from it and denote by $\mathcal{P}$ the remaining packing. We may assume that $\mathcal{P}$ has at least two $2$-disks, as in the opposite case a very rough upper bound on the number of disks in the initial packing would be $f(2) + 6 = 25$, where $6$ is the maximal number of $3$-disks, all touching the $2$-disk. Denote by $C_i$ the number of $i$-disks of $\mathcal{P}$.

Then, consider the union $S$ of open disks of unit radii whose centers are the centers of disks from $\mathcal{P}$ (recall that the disks of the initial packing are of unit diameter, not of unit radius). One can see that centers of the disks we excluded lie on the boundary of $S$.






Theorem~\ref{theorem:f_of_3} follows from the lemmas below.

\begin{lemma}
Let $\gamma$ be a sparse-centered curve with endpoints $a$ and $b$. If $|b - a|\geq 1$ then the length of $\gamma$ is at least $\pi/3$.
\label{lemma:master}
\end{lemma}

\begin{lemma}
\label{lemma:good_curve}
There exists a closed sparse-centered curve $\gamma$ covering the boundary of $S$ and satisfying the inequality
$$C_1 + C_2 + \frac{|\gamma|}{\pi/3} \leq 36,$$
where $|\gamma|$ stands for the length of $\gamma$.
\end{lemma}

\begin{proof}[Deriving Theorem~\ref{theorem:f_of_3} from Lemma~\ref{lemma:master} and Lemma~\ref{lemma:good_curve}]
Consider the curve $\gamma$ from Lemma~\ref{lemma:good_curve}. All centers of excluded disks of the initial packing belong to $\gamma$. Moreover, by Lemma~\ref{lemma:master}, they split $\gamma$ into parts, each of length at least $\pi/3$. Therefore, there is at most $\frac{|\gamma|}{\pi/3}$ excluded disks, which together with Lemma~\ref{lemma:good_curve} implies that the number of disks in the initial packing does not exceed

$$1 + C_1 + C_2 + \frac{|\gamma|}{\pi/3} \leq 37.$$
\end{proof}

\section{Proof of Lemma \ref{lemma:master}}

\begin{observation}
Given points $c_1$, \ldots, $c_n$, consider all sparse-centered curves with centers at a subset of $\{c_i\}$ and with endpoints at least $1$ apart.
There exist such curves of shortest possible length.
\end{observation}

\begin{proof}
%
Let $I$ be the set of the intersection points between all circles of unit radius centered at $c_1$, \ldots, $c_n$. Some of these circles are partitioned into closed arcs by the elements of $I$ (each arc is bounded by two adjacent points from $I$ on its circle). Let $A$ be the set of all these arcs, directed counterclockwise. The set $A$ is obviously finite.

Since any arc of length $\pi/3$ is a valid curve, without loss of generality, we may assume that valid curves have length at most $\pi/3$ and endpoints at least $1$ from each other. Since the set of all arcs $A$ is finite, we may derive that every such curve consists of at most $m$ arcs for some positive $m$.
Using the standard compactness argument, we obtain that there is a shortest valid curve consisting of at most $m$ circular arcs.
\end{proof}

Let $\gamma$ be the sparse-centered curve from Lemma~\ref{lemma:master}, and $t_0$, \ldots, $t_m$ be as in the definition of sparse-centered curves.
Without loss of generality, we may assume that $\gamma$ is a shortest sparse-centered curve satisfying the conditions of Lemma~\ref{lemma:master}.

\begin{observation}
There do not exist $0 < x_1 < x_2 < \ell$ such that the derivatives $\dot{\gamma}(x_1)$ and $\dot{\gamma}(x_2)$ exist and are equal.
\end{observation}

\begin{proof}
Suppose the contrary. Let $p_1$ and $p_2$ be the points $\gamma(x_1)$ and $\gamma(x_2)$, respectively. Let $c_1$ and $c_2$ be the centers of the arcs containing $p_1$ and $p_2$. According to the definition, $\dot{\gamma}(x_1)$ equals $p_1 - c_1$ rotated by the angle $\pi/2$ counterclockwise, and the same holds for $\dot{\gamma}(x_2)$ and $p_2 - c_2$. Since $\dot{\gamma}(x_1) = \dot{\gamma}(x_2)$, we have $p_1 - c_1 = p_2 - c_2$ and, therefore, $c_2 - c_1 = p_2 - p_1$. There are two cases.

\begin{itemize}
    \item Suppose that $c_1 = c_2$. Then the curve $\gamma([0, x_1)\cup[x_2, \ell])\subset\gamma([0, \ell])$ is a shorter sparse-centered curve with endpoints being $1$ apart. This contradicts the minimality of $\gamma$.
    
    \item Otherwise, $|c_2 - c_1| \geq 1$, and hence, $|p_2 - p_1|\geq 1$. But in this case $\gamma([x_1, x_2])\subset\gamma([0, \ell])$ is, again, a sparse-centered curve with length strictly less than $\ell$, which contradicts the minimality of $\gamma$ just as in the first case.
\end{itemize}

In both cases we arrive at a contradiction, which finishes the proof of the observation.
\end{proof}

Denote the set $\left\{\dot{\gamma}(t)\,\colon\,t\in(0, \ell)\setminus\{t_1, \ldots, t_m\}\right\}$ of all directions along $\gamma$ by $S$.
We can introduce a measure $\mu$ on the Borel subsets of the unit circle, assigning $\mu(s)$ to be the length of $\{\gamma(t)\,\colon\,\dot{\gamma}(t)\in s\}$.

\begin{corollary}
$\mu(S) = \ell$.
\end{corollary}

Let us note that $$b - a = \int\limits_0^{\ell}\dot{\gamma}(t)\,\mathrm{d}t = \int\limits_Sv\,\mathrm{d}\mu(v).$$

Let $\proj_{ab}(v)$ be the projection of vector $v$ on the line through $a$ and $b$. Since $\gamma(t)$ is parametrized naturally, $|b - a|$ does not exceed

$$\left|\int\limits_Sv\,\mathrm{d}\mu(v)\right| = \left|\int\limits_S\proj_{ab}(v)\,\mathrm{d}\mu(v)\right| \leq \int\limits_{-\ell/2}^{\ell/2}\cos(\varphi)\,\mathrm{d}\varphi = 2\sin(\ell/2).$$

\medskip Since $|b - a| = 1$, we have $\sin(\ell/2)\geq 1/2$ or $\ell\geq\pi/3$, which finishes the proof of Lemma~\ref{lemma:master}.

\begin{remark}
The requirement of the curve going counterclockwise is crucial. For example, otherwise the curve can be composed of two arcs of length $2\arcsin(1/4)$ each, one going clockwise and the other smoothly continuing the first in the counterclockwise direction. The centers of the corresponding disks are ``at different sides'' of the curve, and, quite clearly, are sufficiently far from each other; see Figure~\ref{fig:incorrect-curve}.
\end{remark}

\begin{figure}[h!]
    \centering
    \includegraphics{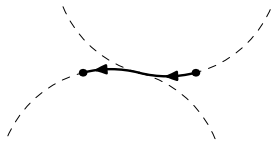}
    \captionsetup{width=.7\textwidth}
    \caption{An example of an invalid curve}
    \label{fig:incorrect-curve}
\end{figure}

\section{Building a curve in Lemma \ref{lemma:good_curve}}

Our current goal is to construct a sparse-centered curve containing the boundary $\partial S$. In Sections~5 and~6 we show that it is sufficiently short.
In the simplest case, we could just use $\partial S$ as such a curve, but in case if $S$ is not simply connected, we cannot do this; see Figure~\ref{fig:partial-s}.

The packing $\mathcal{P}$ of kissing radius $2$ corresponds to a directed plane tree (a planar drawing of a tree) as follows; see Figure~\ref{fig:tree}. The vertices of the tree are the centers of the disks of $\mathcal{P}$. For each disk, except the source, its vertex is connected to the vertex of its parent by an edge, we consider all edges as unit vectors directed from the parent towards its child. The resulting graph is a tree as there is a path from any disk to the source and it has no undirected cycles; otherwise the greatest-layer disk of this cycle has two parents. Since $\mathcal{P}$ is a packing, all edges have unit length and the distance between any two vertices is at least $1$. This implies that no two edges intersect, except, possibly, at their common endpoint. We call this graph the \textit{$\mathcal{P}$-tree}. It can be seen that the $\mathcal{P}$-tree is a subgraph of the minimum-distance graph of the centers of the disks.

\begin{figure}[h!]
    \centering
    \begin{subfigure}[t]{.48\textwidth}
    \includegraphics[width=.95\textwidth]{pics/tree.mps}
    \caption{The $\mathcal{P}$-tree}
    \label{fig:tree}
    \end{subfigure}
    \begin{subfigure}[t]{.48\textwidth}
    \includegraphics[width=.95\textwidth]{pics/traversal.mps}
    \caption{A traversal example. Some disks are counted multiple times, one time per visiting the corresponding vertex}
    \label{fig:traversal}
    \end{subfigure}
    \caption{}
\end{figure}

Let us fix a traversal of the $\mathcal{P}$-tree; see Figure~\ref{fig:traversal}.
Since it is a plane tree, we can consider it as a planar graph with the only face. Traversing the boundary of this face, we get a cyclic ordering $(c_1, c_2, \ldots, c_n)$ of the set of vertices with multiplicities. We assume that all indices are taken modulo $n$. The number of times a vertex occurs in the ordering is its degree\footnote{this ordering is basically a depth-first search (DFS) traversal of the configuration tree.}.
This implies that the center of every $2$-disk occurs exactly once in this order.
Let $(D_1, \ldots, D_n)$ be the corresponding ordering of the disks from $\mathcal{P}$. 
Also, let $B_1$, \ldots, $B_n$ be the open disks of unit radius centered respectively at $c_1$, \ldots, $c_n$. 
Recall that $\bigcup_{i=1}^nB_i = S$.

Given $i,j\in[n]$, call a \textit{subsegment} $\mathcal{D}_{ij}$ (of our traversal) the sequence $(D_i, \ldots, D_j)$, if $i \leq j$, or $(D_i, \ldots, D_n, D_1, \ldots, D_j)$, otherwise. In this paper, we only consider subsegments where $D_i$ and $D_j$ are the only $2$-disks in $\mathcal{D}_{ij}$. For simplicity, we assume that $i\leq j$ in the rest of the paper.

It can be seen that $\mathcal{D}_{ij}$ consists of either $3$ or $5$ disks. Indeed, if $\parent(D_i) = \parent(D_j)$, or, equivalently, the source $D$ is not in $\mathcal{D}_{ij}$, then $\mathcal{D}_{ij} = (D_i, \parent(D_i), D_j)$. Otherwise, $$\mathcal{D}_{ij} = \left(D_i, \parent(D_i), D, \parent(D_j), D_j\right).$$

We assume that the $0$-disk $D$ is centered at the origin $c$. By analogy with $D$, denote by $B$ the disk of unit radius centered at $c$. For any $i$ such that $D_i$ is a $2$-disk, denote by $f_i$ the farthest point of the disk $B_i$ from the origin $c$.
Formally,

$$f_i = c_i\cdot\left(1 + \frac{1}{|c_i|}\right).$$

By $[ab)$ we denote the ray starting at point $a$ and passing through $b$. If $D_i$ and $D_j$ are two consecutive occurrences of $2$-disks, consider segments $c_ic_{i+1}$, \ldots, $c_{j-1}c_j$, and two rays $[c_if_i)$ and $[c_jf_j)$. Their union divide the plane into two parts; denote by $R_{ij}$ the closed one for which $(c_i, \ldots, c_j)$ is the clockwise ordering of vertices.
As we already know, each region contains either $3$ or $5$ vertices.

The plane is then partitioned into such regions (see Figure~\ref{fig:division-into-regions}), and for each of them we will construct a sparse-centered curve $\gamma_{ij}\subset R_{ij}$ with centers among $\{c_i, \ldots, c_j\}$, having endpoints $f_i$ and $f_j$ and containing $\partial S\cap R_{ij}$. After this, we consider $\gamma$ as the union of all $\gamma_{ij}$.

\begin{figure}[h!]
    \centering
    \begin{subfigure}{.4\textwidth}
    \includegraphics[width=\textwidth]{pics/fig.mps}
    \caption{$\partial{S}$ is in bold}
    \label{fig:partial-s}
    \end{subfigure}
    \begin{subfigure}{.4\textwidth}
    \includegraphics[width=\textwidth]{pics/fig_reg.mps}
    \caption{Division into regions}
    \label{fig:division-into-regions}
    \end{subfigure}
    \caption{}
    \label{fig:regions}
\end{figure}



First, we prove the following auxiliary lemmas.

\begin{observation}
If $D_i$ is a $2$-disk then no point of the ray $[c_if_i)$ outside the segment $c_if_i$ belongs to any disk. In particular, $f_i\in\partial S$.
\label{lemma:far}
\end{observation}

\begin{proof}
Let $g_i$ be a point of the ray $[c_if_i)$ outside the segment $c_if_i$. Suppose that $g_i$ lies in $S$.
Assume that $g_i$ is covered by some disk $D'_k$, and let $D = D'_0$, \ldots, $D'_k$ ($1\leq k\leq 2$) be the disks corresponding to the path in the $\mathcal{P}$-tree from the origin to the center of $D'_k$; denote by $c'_j$ the center of $D'_j$.
Since for any $j\in[k - 1]$, all sides of the triangle $c'_jc'_{j+1}c_i$ have lengths at least $1$ and $|c'_{j+1} - c'_j| = 1$, the angle between the vectors $c'_j - c_i$ and $c'_{j+1} - c_i$ is at most $\pi/3$. Similarly, the angle between $c'_k - c_i$ and $g_i - c_i$ is strictly less than $\pi/3$. Since $k\leq 2$, these inequalities contradict the fact that the angle between $c'_0 - c_i$ and $g_i - c_i$ equals $\pi$. Thus, $g_i\notin S$.

Since all points of the segment $c_if_i$ but $f_i$ clearly lie in $B_i\subset S$, we have $f_i\in\partial{S}$.

Here we proved, in particular, that for any $c_k$ the angle between $c_i$ and $c_k - c_i$ does not exceed $2\pi/3$. We will use this in Lemma~\ref{lemma:E-is-connected}.
\end{proof}


\begin{observation}
Let $\mathcal{D}_{ij}$ be a subsegment. Let $S_{ij}$ be the union of $B_i$, \ldots, $B_j$. Then $S\cap R_{ij} = S_{ij}\cap R_{ij}$; see Figure~\ref{fig:only_inner}.

\label{lemma:only_inner}
\end{observation}

\begin{figure}[h!]
    \centering
    \includegraphics{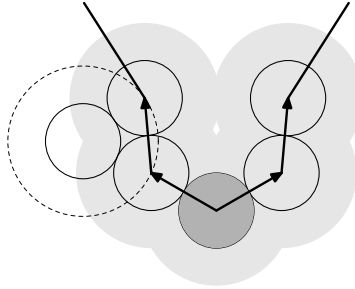}
    \captionsetup{width=.7\textwidth}
    \caption{The part of the boundary of the dashed disk in the region is completely in the gray area, and therefore, add nothing to the border}
    \label{fig:only_inner}
\end{figure}


\begin{proof}
Suppose that the set $(S\setminus S_{ij})\cap R_{ij}$ is not empty and $p$ is any point of this set.
Let $c'$ be the center of any disk containing the point $p$. Since $c'\not\in R_{ij}$ and $p\in R_{ij}$, the segment $c'p$ intersects $\partial{R_{ij}}$. Consider two cases.

If $c'p$ intersects some segment $c_kc_{k+1}$ then by the triangle inequality one of the segments $c_kp$ and $c_{k+1}c'$ is strictly shorter than $1$, which cannot be the case.

Otherwise $c'p$ intersects one of the rays $[c_if_i)$ and $[c_jf_j)$. Without loss of generality assume that it is $[c_if_i)$. By Observation~\ref{lemma:far} one can assume that the intersection is on the segment $c_if_i$. Analogously, at least one of the segments $c_ip$ and $f_ic'$ is strictly shorter than $1$, which also leads to a contradiction.
\end{proof}


Denote $\partial S\cap R_{ij}$ by $\sigma_{ij}$.
We say that a disk $B_k$ is \textit{involved in $\sigma_{ij}$} if $\partial{B_k}\cap \sigma_{ij}$ is nonempty; that is, if there is a part of its boundary in $R_{ij}$ that is not covered by interiors of other disks. 
Note that $B_i$ and $B_j$ are always involved in $\sigma_{ij}$.
This follows from the fact that $f_i$ and $f_j$ are always points of the boundary.

\begin{claim}\label{lemma:can-construct-gamma-ij}
Let $\mathcal{D}_{ij}$ be a subsegment. Let $i = k_1 < \ldots < k_l = j$ be a sequence of indices of all disks from $\mathcal{D}_{ij}$ that are involved in $\sigma_{ij}$. Let $\gamma_{ij}$ be the curve constructed in the following way:
\begin{enumerate}
    \item $\gamma_{ij}$ starts at $f_i$ and goes counterclockwise along $\partial B_i$ until meeting $\partial B_{k_2}\cap R_{ij}$,
    \item Next, $\gamma_{ij}$ goes counterclockwise along $\partial B_{k_2}$ until meeting $\partial B_{k_3}$,
    \item Next, it goes in the following manner along $B_{k_3}$, \ldots, $B_{k_{l-1}}$,
    \item Finally, $\gamma_{ij}$ goes along $\partial B_j$, ending at $f_j$.
\end{enumerate}
Let $\gamma$ be the concatenation of all $\gamma_{ij}$. Then $\gamma$ is a well-defined sparse-centered curve containing $\partial S$.
\end{claim}

The proof is given in Appendix~\ref{section:proof-of-claim}.

\section{The local inequality implies Lemma~\ref{lemma:good_curve}}

The goal of this section is to formulate an upper bound on the length of each $\gamma_{ij}$ and to show how it implies Lemma~\ref{lemma:good_curve}.

\subsection{Angular notation}

Recall that the edges of the $\mathcal{P}$-tree are directed from parents. Consider these edges as unit vectors and denote by $E$ the set of all these edges.
For any $2$-disk $D_i$ let $v_i = c_{i+1} - c_i$ and $u_i = c_{i+1} - c$.

We define two functions $\angle(\cdot, \cdot)\colon E^2\to\mathbb{R}$ and $\angccw(\cdot, \cdot)\colon (\mathbb{R}^2\setminus\{c\})^2\to[0, 2\pi)$. Essentially, each of them is a directed angle between two vectors; that is, the angle the first vector needs to be rotated by counterclockwise in order to become the positive scalar multiple of the second one. However, we need to define the range of values for $\angle(\cdot, \cdot)$ and we do it in the following way (here $i$ and $j$ are indices of some $2$-disks).

\begin{enumerate}[label={\{\arabic*\}}]
\item $\angle(u_i, u_j)\in[0, 2\pi)$; \label{rule:uu}
\item $\angle(u_i, v_i)\in(-\pi, \pi)$; \label{rule:uivi} 
\item $\angle(v_i, u_i) = -\angle(u_i, v_i)$; \label{rule:viui}
\item $\angle(v_i, u_j) = \angle(v_i, u_i) + \angle(u_i, u_j)$, and $\angle(u_j, v_i) = \angle(u_j, u_i) + \angle(u_i, v_i)$; \label{rule:uivj}
\item $\angle(v_i, v_j) = \angle(v_i, u_i) + \angle(u_i, u_j) + \angle(u_j, v_j)$. \label{rule:vv}
\end{enumerate}

The $\angle(\cdot, \cdot)$ operator is clearly well-defined.



\begin{claim}
\label{lemma:can-construct}

Let $i$ and $j$ be two consecutive $2$-disks and let $k$ be the number of occurrences of the source in $\mathcal{D}_{ij}$. Denote $\angle(v_i, v_j)$ by $\varphi_{ij}$ and $\angccw(c_i, c_j)$ by $\alpha_{ij}$. If $k = 1$, also denote $\angle(u_i, u_j)$ by $\psi_{ij}$. Then we have

\begin{equation*}
\label{eq:curve_bound}
|\gamma_{ij}|\leq\begin{cases}
\displaystyle
\varphi_{ij} + \alpha_{ij}, & \text{when }k = 0, \\
3\psi_{ij} - \frac{2\pi}{3} + 2\varphi_{ij} + \alpha_{ij}, & \text{when }k = 1.
\end{cases}
\end{equation*}
\end{claim}


If we denote by $\mathcal{R}$ the set of all regions $R_{ij}$, then we obtain \begin{equation}\sum_{R_{ij}\in\mathcal{R}}\varphi_{ij} = \sum_{R_{ij}\in\mathcal{R}}\alpha_{ij} = \sum_{R_{ij}\in\mathcal{R}}\psi_{ij} = 2\pi.\label{eqn:sums-of-angles-are-2pi}
\end{equation}
Indeed, the vectors $u_i$ divide $2\pi$ into angles $\psi_{ij}$, and the rays from $c$ through the centers of $2$-disks divide $2\pi$ into angles $\alpha_{ij}$.

To establish the equality for $\varphi_{ij}$, let $v_{i_1}$, \ldots, $v_{i_m}$ be the vectors outgoing from any $1$-disk, and let $u = u_{i_1} = \ldots = u_{i_m}$ be the vector from the source to it.
We have
\begin{align*}
\angle(v_{i_{m-1}}, v_{i_m}) + \angle(v_{i_m}, u) & \stackrel{\ref{rule:vv}}{=} \big(\angle(v_{i_{m-1}}, u) + \angle(u, u) + \angle(u, v_{i_m})\big) + \angle(v_{i_m}, u) \\ & \stackrel{\ref{rule:viui}}{=} \angle(v_{i_{m-1}}, u).
\end{align*}
Therefore, we obtain
\begin{multline*}
\angle(u, v_{i_1}) + \angle(v_{i_1}, v_{i_2}) + \ldots + \angle(v_{i_{m-1}}, v_{i_m}) + \angle(v_{i_m}, u) \\
= \angle(u, v_{i_1}) + \angle(v_{i_1}, v_{i_2}) + \ldots + \angle(v_{i_{m-2}}, v_{i_{m-1}}) + \angle(v_{i_{m-1}}, u).
\end{multline*}
Thus, by induction,
$$\angle(u, v_{i_1}) + \angle(v_{i_1}, v_{i_2}) + \ldots + \angle(v_{i_{m-1}}, v_{i_m}) + \angle(v_{i_m}, u) = 0.$$


One can see that after expanding $\sum\varphi_{ij}$ by the definition (see~\ref{rule:vv}) and subtracting the left hand side of the equality above for all $1$-disks, we obtain exactly $\sum\psi_{ij}$. Indeed, after canceling everything out, the only summands that remain are $\angle(u_i, u_j) = \psi_{ij}$.

\begin{proof}[Claim~\ref{lemma:can-construct} implies Lemma~\ref{lemma:good_curve}]


%
Denote by $\mathcal{R}_0$ the set of all regions $R_{ij}$, for which the source does not occur in $\mathcal{D}_{ij}$, that is, $\mathcal{D}_{ij}$ consists of three disks. Denote the set of all other regions by $\mathcal{R}_1$. In particular, $\mathcal{R} = \mathcal{R}_0\cup\mathcal{R}_1$ and $\varphi_{ij}\geq\pi/3$ for $R_{ij}\in\mathcal{R}_0$. Also, $|\mathcal{R}_1| = C_1$ and $|\mathcal{R}_0| = C_2 - C_1$.
Let us bound the length of $\gamma$ as the sum of lengths $\gamma_{ij}$ using Claim~\ref{lemma:can-construct}.
\begin{align*}
    |\gamma| &\leq \sum_{R_{ij}\in\mathcal{R}_0}(\varphi_{ij} + \alpha_{ij}) + \sum_{R_{ij}\in\mathcal{R}_1}\left(3\psi_{ij} - \frac{2\pi}{3} + 2\varphi_{ij} + \alpha_{ij}\right)  \\
    &= 2\sum_{R_{ij}\in\mathcal{R}}\varphi_{ij} + \sum_{R_{ij}\in\mathcal{R}}\alpha_{ij} + 3\sum_{R_{ij}\in\mathcal{R}}\psi_{ij} - \frac{2\pi}{3}\cdot C_1 - \sum_{R_{ij}\in\mathcal{R}_0}\varphi_{ij}  \\
    &\stackrel{\eqref{eqn:sums-of-angles-are-2pi}}{\leq} 2\cdot 2\pi + 2\pi + 3\cdot 2\pi - \frac{2\pi}{3}\cdot C_1 - \frac{\pi}{3}\cdot (C_2 - C_1)  \\
    &= 12\pi - \frac{\pi}{3}(C_1 + C_2).
\end{align*}

The last inequality is equivalent to $C_1 + C_2 + \frac{|\gamma|}{\pi/3}\leq 36$.
\end{proof}


\subsection{Direction jump}

Instead of proving Claim~\ref{lemma:can-construct}, first we simplify its statement. For that, we introduce the concept of direction jump.

Let the sparse-centered curve $\gamma_{ij}$ be composed of $m$ arcs. Let $w_k^s$ and $w_k^f$ be the unit vectors from the center of the $k\textsuperscript{th}$ arc to its beginning and its end, respectively. In particular, $w_1^s = f_i - c_i$ and $w_m^f = f_j - c_j$. By the \emph{direction jump} between the $k\textsuperscript{th}$ and the $(k+1)\textsuperscript{th}$ arcs or the \emph{$k\textsuperscript{th}$ direction jump}, we define $\angccw(w_{k+1}^s, w_k^f)$; see Figure~\ref{fig:direction-jumps}. We denote the $k\textsuperscript{th}$ direction jump by $\dj_k$. The sum of all directed jumps $\sum_{k=1}^{m-1}\dj_k$ is denoted by $\Dj_{ij}$.

For example, on Figure~\ref{fig:direction-jumps} we have $\dj_1 = \dj_3 = \pi/3$, and $\dj_2 = \psi_{ij}$.

If $B_{k_1}$ and $B_{k_2}$ are two disks, we will call the \emph{direction jump between $B_{k_1}$ and $B_{k_2}$} the direction jump between the arcs of these disks which are present in the curve $\gamma_{ij}$. 

\begin{figure}[h!]
    \centering
    \includegraphics[width=.6\textwidth]{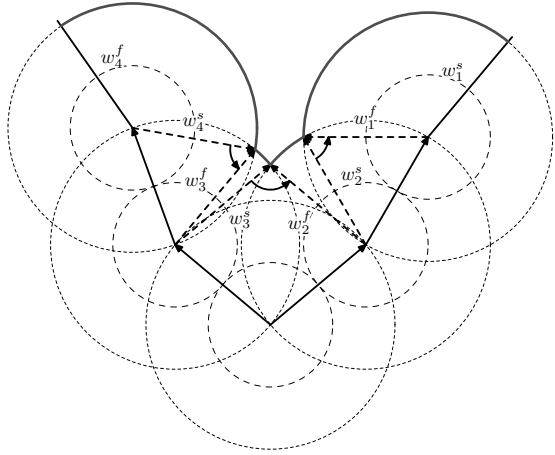}
    \caption{Direction jumps}
    \label{fig:direction-jumps}
\end{figure}

\begin{observation}
For the curve $\gamma_{ij}$ the following holds.
$$|\gamma_{ij}| = \Dj_{ij} + \alpha_{ij}.$$
\end{observation}

\begin{proof}
Let $x = c_i$ and $y = f_i$. We are going to continuously move points $x$ and $y$ to $c_j$ and $f_j$, respectively. This process will be separated in $(2m-1)$ steps. We will track the rotation of the vector $y - x$ during each step and add these rotations up. We will also fix a direction and ensure that at no point of time during the process does $y - x$ aim at that direction. This will imply that the total rotation equals $\alpha_{ij}$, while we know a priori that the total rotation must equal $\alpha_{ij} + 2\pi k$ for some integer $k$.

The process looks as follows. At the first step we rotate $y$ counterclockwise about $x$, until $y - x = w_1^f$. During the second step, we rotate $x$ clockwise until $y - x = w_2^s$, and so on. Formally, at the $(2l-1)\textsuperscript{th}$ step we move $y$ counterclockwise in such a way that $y - x$ becomes $w_l^f$, and at the $2l\textsuperscript{th}$ step we move $x$ clockwise until $y - x = w_{l+1}^s$.

It is clear that each $(2k-1)\textsuperscript{th}$ step rotates $y - x$ by $\angccw(w_{l}^s, w_{l}^f)$, and each $2l\textsuperscript{th}$ step rotates it by $-\angccw(w_{l+1}^s, w_{l}^f)$. Therefore, if we know that the total rotation equals $\alpha_{ij}$, we show that
\begin{align*}
\alpha_{ij} & = \angccw(w_{1}^s, w_{m}^f) \\
& = \sum_{l=1}^{m}\angccw(w_{l}^s, w_{l}^f) - \sum_{l=1}^{m-1}\angccw(w_{l+1}^s, w_{l}^f) \\
& = |\gamma_{ij}| - \sum_{l=1}^{m-1}\dj_{l} = |\gamma_{ij}| - \Dj_{ij}.
\end{align*}

To prove that there is a direction which $y - x$ never aims at during the process, let us state a couple of inequalities. In particular, we claim the following.
\begin{itemize}
    \item If an arc of $B_i$ is present in $\gamma_{ij}$, then for each its point $p$ there holds $\angccw(f_i - c_i, p - c_i)\leq\pi$.
    \item If an arc of $B_{i+1}$ is present in $\gamma_{ij}$, then for each its point $p$ there holds $\angccw(c_i, p - c_{i+1})\leq\pi$.
    \item If an arc of $B_j$ is present in $\gamma_{ij}$, then for each its point $p$ there holds $\angccw(p - c_j, f_j - c_j)\leq\pi$.
    \item If an arc of $B_{j-1}$ is present in $\gamma_{ij}$, then for each its point $p$ there holds $\angccw(p - c_{j-1}, c_j)\leq\pi$.
    \item If an arc of $B_{i+2}=B_{j-2}$ is present in $\gamma_{ij}$, then for each its point $p$ there holds $\angccw(c_i, p) + \angccw(p, c_j) = \angccw(c_i, c_j) = \alpha_{ij}$.
\end{itemize}
If we prove this, then all directions during the process belong to the union of $\{p\,\colon\,\angccw(c_i, p)\leq\pi\}$ and $\{p\,\colon\,\angccw(p, c_j)\leq\pi\}$, which do not cover the set of all possible directions, as $\angccw(c_i, c_j)\neq 0$.

Since $|c_i - c|\geq |c_i - c_{i+1}| = |c_{i+1} - c|$, if $\angccw(f_i - c_i, c_{i+1} - c_i) > \pi$, then we have the inequality
$$\angccw(f_i - c_i, c_{i+1} - c_i) = \angccw(f_i - c_i, c - c_i) + \angccw(c - c_i, c_{i+1} - c_i)\leq\frac{4\pi}{3}.$$
As shown in the proof of Claim~\ref{lemma:can-construct-gamma-ij} (specifically, in Lemma~\ref{lemma:delaunay-covers-neighbors}), for any point $p$ of the arc of $B_i$ in $\gamma_{ij}$ we have $\angccw(p - c_i, c_{i+1} - c_i)\geq\pi/3$.
But then it follows from the definition of $\gamma_{ij}$ that
$$\angccw(f_i - c_i, p - c_i) = \angccw(f_i - c_i, c_{i+1} - c_i) - \angccw(p - c_i, c_{i+1} - c_i)\leq\frac{4\pi}{3} - \frac{\pi}{3} = \pi.$$
Similarly, if an arc of $B_j$ is present, then for each its point $p$ there holds $\angccw(p - c_j, f_j - c_j)\leq\pi$.

If $\angccw(c_i - c_{i+1}, c_{i+2} - c_{i+1})\leq\pi$, then for any point $p$ of the arc of $B_{i+1}$ in $\gamma_{ij}$ we have $\angccw(c_i, p - c_{i+1})\leq\angccw(c_i - c_{i+1}, c_{i+2} - c_{i+1})\leq\pi$. Suppose that $\angccw(c_i - c_{i+1}, c_{i+2} - c_{i+1}) > \pi$. Again, we have the inequalities $\angccw(c_i - c_{i+1}, p - c_{i+1})\geq\pi/3$ and $\angccw(p - c_{i+1}, c - c_{i+1})\leq\pi/3$. Since $|c_i - c|\geq |c_i - c_{i+1}| = |c_{i+1} - c|$, we have the inequality $\angccw(c_i - c_{i+1}, c_i - c)\leq\pi/3$. This means that
\begin{align*}
\frac{\pi}{3} - \frac{\pi}{3} & \leq\angccw(c_i - c_{i+1}, p - c_{i+1}) - \angccw(c_i - c_{i+1}, c_i) \\
& = \angccw(c_i, c - c_{i+1}) - \angccw(p - c_{i+1}, c_{i+1})\leq\frac{4\pi}{3} - \frac{\pi}{3},
\end{align*}
which implies that $\angccw(c_i, p - c_{i+1})\mod{2\pi}\in[0, \pi]$, hence, $\angccw(c_i, p - c_{i+1})\leq\pi$. Analogously, if an arc of $B_{j-1}$ is present in $\gamma_{ij}$, then for each its point $p$ there holds $\angccw(p - c_{j-1}, c_j)\leq\pi$.

Suppose that an arc of $B_{i+2}=B_{j-2}$ is present in $\gamma_{ij}$, and let $p$ be any of its points. The equality $\angccw(c_i, p) + \angccw(p, c_j) = \angccw(c_i, c_j)$ is equivalent to the fact that $p$ is located on the arc of $B_{i+2}$, going from $c_i/|c_i|$ counterclockwise until the point $c_j/|c_j|$. Since $p$ lies on the arc of $B_{i+2}$ from $c_{i+1}$ to $c_{j-1}$ by the definition, we know that $\angccw(c_{i+1}, p) + \angccw(p, c_{j-1}) = \angccw(c_{i+1}, c_{j-1})$. Due to Lemma~\ref{lemma:delaunay-covers-neighbors}), we know that $\angccw(c_{i+1}, p)\geq\pi/3$ and $\angccw(p, c_{j-1})\geq\pi/3$. Since $|c_i - c|\geq|c_i - c_{i+1}|=|c_{i+1} - c|$ and $|c_j - c|\geq|c_j - c_{j-1}|=|c_{j-1} - c|$, we know that the smaller angle between $c_i - c$ and $c_{i+1} - c$ does not exceed $\pi/3$, as well as the smaller angle between $c_j - c$ and $c_{j-1} - c$. This means that to obtain the arc of $B_{i+2}$ going from $c_i/|c_i|$ counterclockwise until the point $c_j/|c_j|$ from the arc of $B_{i+2}$ from $c_{i+1}$ to $c_{j-1}$, one can prolong or shorten it by at most $\pi/3$ in both directions, which leaves the point $p$ on it.
\end{proof}




This observation implies that in order to prove Claim~\ref{lemma:can-construct} it suffices to show the following assertion.

\begin{claim}\label{lemma:can-construct-without-alpha} We have
$$\Dj_{ij}\leq\begin{cases}\varphi_{ij}, & \textrm{if }k = 0, \\ 3\psi_{ij} - \frac{2\pi}{3} + 2\varphi_{ij}, & \textrm{if }k = 1. \end{cases}$$
\end{claim}

\section{Proof of Claim~\ref{lemma:can-construct-without-alpha}}

For simplicity, we use $\psi$, $\varphi$, and $\Dj$ instead of $\psi_{ij}$, $\varphi_{ij}$, and $\Dj_{ij}$ in this section.

\begin{observation}\label{lemma:not-covered-2pi-3}
If a disk $B_l$ is involved in $\sigma_{ij}$ for $i < l < j$, then $\angccw(c_{l-1} - c_l, c_{l+1} - c_l)\geq 2\pi/3$.
\end{observation}

\begin{proof}
Consider any point $p$ of $\partial{B_l}\cap\sigma_{ij}$. As $p$ is not inside $B_{l+1}$, we have $|p - c_{l+1}|\geq 1$, or $\angccw(p - c_l, c_{l+1} - c_l)\geq\pi/3$. Similarly, $\angccw(c_{l-1} - c_l, p - c_l)\geq\pi/3$. Hence, $\angccw(c_{l-1} - c_l, c_{l+1} - c_l)\geq2\pi/3$.
\end{proof}

\subsection{Case $k = 0$}

There are two $2$-disks $D_i$ and $D_j$ and one $1$-disk $D_{i+1} = D_{j-1}$ occurring in $\mathcal{D}_{ij}$.

\begin{enumerate}[label={\bf Case \arabic*: }, wide, labelwidth=!, labelindent=0pt]
\caseitem{0a}{$B_{i+1} = B_{j-1}$ is not involved in $\sigma_{ij}$}

In this case there is the only direction jump of size $\varphi$. 

\caseitem{0b}{$B_{i+1} = B_{j-1}$ is involved in $\sigma_{ij}$}

In this case there are two direction jumps, each of them equal to $\pi/3$. Due to Observation~\ref{lemma:not-covered-2pi-3}, $\varphi\geq 2\pi/3$, which finishes the proof.


\subsection{Case $k = 1$, general observations}

There are two $2$-disks $D_i$ and $D_j$, two $1$-disks $D_{i+1}$ and $D_{j-1}$, and one $0$-disk $D$ occurring in $\mathcal{D}_{ij}$.

\begin{observation}\label{lemma:direction-jumps-are-bounded}
The direction jump between disks $B_{i_1}$ and $B_{i_2}$ does not exceed $(i_2 - i_1)\frac{\pi}{3}$.
\end{observation}

\begin{proof}
Let $p$ be the point of the direction jump from $B_{i_1}$ to $B_{i_2}$. Then, as for every $l$ such that $i_1 < l < i_2$ we have $|p - c_l|\geq 1$, the vertex $p$ is always at the smallest angle of any triangle $c_lc_{l+1}p$ for $i_1\leq l < i_2$. Since the angle about $p$ is at most $\pi/3$ in every such triangle, summing up all these inequalities gives us the required bound.
\end{proof}

\begin{observation}
If $\psi\geq\pi$, then Claim~\ref{lemma:can-construct-without-alpha} holds.
\end{observation}

\begin{proof}
Note that $\angle(u_i, v_i)\in\left[-\frac{2\pi}{3}, \frac{2\pi}{3}\right]$. Indeed, otherwise the angle between $c_{i+1} - c_i$ and $c_{i+1} - c$ is less than $\pi/3$, which implies that $|c_i - c| < 1$. Similarly, $\angle(u_j, v_j)\in\left[-\frac{2\pi}{3}, \frac{2\pi}{3}\right]$.

But then
$$\varphi \stackrel{\ref{rule:vv}}{=} \angle(v_i, u_i) + \angle(u_i, u_j) + \angle(u_j, v_j) \geq -\frac{2\pi}{3} + \psi - \frac{2\pi}{3} \geq -\frac{\pi}{3}.$$

Thus,
$3\psi - \frac{2\pi}{3} + 2\varphi\geq \frac{5\pi}{3}.$

On the other hand, $\Dj\leq\frac{4\pi}{3}$. Indeed, if $i = i_1$, $i_2$, \ldots, $i_m = j$ are all involved in $\sigma_{ij}$ disks, then, by Observation~\ref{lemma:direction-jumps-are-bounded}, we have $$\Dj\leq(i_2 - i_1)\frac{\pi}{3} + \ldots + (i_m - i_{m-1})\frac{\pi}{3} = \frac{4\pi}{3},$$
which concludes the proof.
\end{proof}

From now on, we assume $\psi < \pi$.

\begin{observation} \label{lemma:inside-the-rhombus}
If $c'$ is the point such that $c_{i+1}cc_{j-1}c'$ is a parallelogram, then $c'\in R_{ij}$. Let $t$ be any point such that $|t - c_{i+1}|\geq 1$ and $|t - c_{j-1}| \geq 1$. Then $t$ cannot be inside the rhombus $c_{i+1}cc_{j-1}c'$.
\end{observation}

\begin{proof}
The triangle $c_{i+1}cc'$ lies in $B_{i+1}$, and the triangle $c_{j-1}cc'$ lies in $B_{j-1}$. If $c'\notin R_{ij}$, then there is a unit segment or a ray that is a side of $R_{ij}$ and that separates $c'$ from $c$ in this region. 
It is obvious that it cannot be any of the rays, as both of them belong to the line containing $c$. Similarly, it can be neither $cc_{i+1}$ nor $cc_{j-1}$. Hence, the only remaining options are segments $c_ic_{i+1}$ and $c_jc_{j-1}$. Suppose that, say, the segment $c_ic_{i+1}$ intersects the segment $cc'$.
This means that the point $|c_{j-1} - c_i| < 1$, because in the triangle $c_ic_{i+1}c_{j-1}$ the angle about $c_{j-1}$ is greater than the angle about $c_{i+1}$, and $|c_{i+1} - c_i| = 1$.
Similarly, $c_jc_{j-1}$ cannot separate $c$ from $c'$. Thus, $c'\in R_{ij}$.

Suppose that $t$ belongs to one of the triangles $cc_{i+1}c'$ and $cc_{j-1}c'$, say, the first one. But since $|c_{i+1} - c| = |c_{i+1} - c'| = 1$, any other point of between $c$ and $c'$ is strictly less than $1$ away from $c_{i+1}$; therefore, so is point $t$. See Figure~\ref{fig:phi-and-psi} for clarity.
\end{proof}

\begin{observation} \label{lemma:psi-plus-phi} We have
\begin{align*}
\psi + \varphi & = \angle(u_i, v_j) + \angle(v_i, u_j) \\
& = \angccw(c_i - c_{i+1}, c' - c_{i+1}) + \angccw(c' - c_{j-1}, c_j - c_{j-1}).    
\end{align*}

\end{observation}

\begin{proof}
Note that neither $c_i$ nor $c_j$ is inside the rhombus $c_{i+1}cc_{j-1}c'$, according to Observation~\ref{lemma:inside-the-rhombus}. Thus, $\angle(v_i, u_j)\geq 0$ and $\angle(u_i, v_j)\geq 0$. Hence,
\begin{align*}
\psi + \varphi & = \angle(v_i, v_j) + \angle(u_i, u_j) \\
&\stackrel{\ref{rule:vv}}{=} \big(\angle(v_i, u_i) + \angle(u_i, u_j) + \angle(u_j, v_j)\big) + \angle(u_i, u_j) \\
&= \big(\angle(v_i, u_i) + \angle(u_i, u_j)\big) + \big(\angle(u_i, u_j) + \angle(u_j, v_j)\big) \\
&\stackrel{\ref{rule:uivj}}{=} \angle(v_i, u_j) + \angle(u_i, v_j) \\
&= \angccw(c_i - c_{i+1}, c' - c_{i+1}) + \angccw(c' - c_{j-1}, c_j - c_{j-1}),
\end{align*}
which finishes the proof.
\end{proof}

\begin{figure}[h!]
    \centering
    \begin{subfigure}[t]{.4\textwidth}
    \includegraphics[width=\textwidth]{pics/phi-and-psi.mps}
    \caption{Curve from Obsevation~\ref{lemma:obvious-stuff-about-psi}}
    \label{fig:phi-and-psi}
    \end{subfigure}
    \begin{subfigure}[t]{.4\textwidth}
    \includegraphics[width=\textwidth]{pics/only-middle-involved.mps}
    \caption{Curve from Obsevation~\ref{lemma:involved-source}}
    \label{fig:only-middle-involved}
    \end{subfigure}
    \caption{}
\end{figure}

\begin{observation} \label{lemma:obvious-stuff-about-psi}
$\psi + \varphi\geq \pi/3$ and $\psi\geq \pi/3$.
\end{observation}

\begin{proof}
Since $|c_{i+1} - c_{j-1}| \geq 1$, we have $\psi\geq\pi/3$. Consider the sparse-centered curve that starts at $c_i$, follows the perimeter of $B_{i+1}$ counterclockwise until it reaches $c'$, then switches to the perimeter of $B_{j-1}$ until it reaches $c_j$; see Figure~\ref{fig:phi-and-psi}. Since it satisfies the assumption of Lemma \ref{lemma:master}, its length is at least $\pi/3$. On the other hand, according to Observation~\ref{lemma:psi-plus-phi}, its length is exactly $\psi + \varphi$.
\end{proof}

\begin{observation}
If $B$ is involved in $\sigma_{ij}$, then $\varphi\geq 0$. \label{lemma:involved-source}
\end{observation}

\begin{proof}
It follows from the constraints that the arc of $B$ that goes from $c_{i+1}$ to $c_{j-1}$ counterclockwise and which has length $\psi$ is not completely covered by other disks; in particular, it is not fully covered by $B_i$ and $B_j$.
Note that if $\angle(u_i, v_i) > 0$, then $B_i$ covers the arc of length $\angle(u_i, v_i)$ of $B$, starting at $c_{i+1}$. In other words, $B_i$ covers an arc of length $\max(\angle(u_i, v_i), 0)$, starting at $c_{i+1}$. Similarly, $B_j$ covers an arc of length $\max(\angle(v_j, u_j), 0)$, ending at $c_{j-1}$, see Figure~\ref{fig:only-middle-involved}.


Since $B_{i+2}$ is involved in $\sigma_{ij}$, we have
\begin{align*}
    \psi & \geq\max(\angle(u_i, v_i), 0) + \max(\angle(v_j, u_j), 0) \\
    & \geq\angle(u_i, v_i) + \angle(v_j, u_j) \\
    & \stackrel{\ref{rule:uivi}}{=} \angle(u_i, u_j) + \angle(v_j, v_i) = \psi - \varphi,
\end{align*}
which implies that $\varphi\geq 0$.\end{proof}

\begin{corollary}\label{lemma:middle-involved-side-angles}
If $B$ is involved in $\sigma_{ij}$, then
$$\angccw(c_i - c_{i+1}, c - c_{i+1}) + \angccw(c - c_{j-1}, c_j - c_{j-1})\leq 3\psi - \frac{2\pi}{3} + 2\varphi.$$
\end{corollary}

\begin{proof}
According to Observation~\ref{lemma:involved-source}, $\varphi\geq 0$. Also, due to Observation~\ref{lemma:not-covered-2pi-3}, $\psi\geq2\pi/3$. But then
\begin{align*}
    & \angccw(c_i - c_{i+1}, c - c_{i+1}) + \angccw(c - c_{j-1}, c_j - c_{j-1}) \\
    & = (\pi - \angle(u_i, v_i)) + (\pi - \angle(v_j, u_j)) \\
    & = 2\pi - \psi + \varphi \\
    & \leq 2\pi - \psi + \varphi + 4\left(\psi - \frac{2\pi}{3}\right) + \varphi = 3\psi - \frac{2\pi}{3} + 2\varphi.
\end{align*}
\end{proof}

\subsection{Case $k = 1$, finishing the proof}

Let $I$ be the set of disks from $\{B_{i+1}, B, B_{j-1}\}$ which are involved in $\sigma_{ij}$.

\caseitem[subsec:case-000]{1a}{$I = \varnothing$}

There is the only direction jump, and it is between the disks $B_i$ and $B_j$. Denote this direction jump by $\dj$. As $\psi\geq\pi/3$ by Observation~\ref{lemma:obvious-stuff-about-psi}, it suffices to prove that $\dj\leq 2\psi - \frac{\pi}{3} + 2\varphi$.

If $\psi + \varphi\geq 2\pi/3$, then, obviously, $\dj\leq \pi = 2\cdot\frac{2\pi}{3} - \frac{\pi}{3}\leq 2(\psi + \varphi) - \frac{\pi}{3}$. Hence, one may safely assume that $\psi + \varphi\leq 2\pi/3$, which together with the inequality $\psi\geq\pi/3$ gives us

\begin{equation}
\label{eq:straighten}
\psi \geq \frac{\pi}{3}\geq \psi + \varphi - \frac{\pi}{3}\geq \varphi.
\end{equation}

By the triangle inequality, we obtain
$$2\sin\frac{\dj}{2} = |c_i - c_j| = |u_i + v_i - u_j - v_j|\leq |u_i - u_j| + |v_i - v_j| = 2\sin\frac{\psi}{2} + 2\sin\frac{\varphi}{2}.$$ Here we use that in a triangle with two unit sides and angle $\beta$ between them the third side has length $2\sin(\beta/2)$.

If $s = \psi + \varphi\in\left[\frac{\pi}{3}, \frac{2\pi}{3}\right]$ is fixed, then the right hand side is maximized when $|\psi - \varphi|$ is minimized. Indeed, if $z = e^{i\psi/2} + e^{i\varphi/2}$, then $\arg{z} = s/4$; therefore maximizing $\Im{z}$ means maximizing $|z|$, or minimizing the angle between $e^{i\psi/2}$ and $e^{i\varphi/2}$. According to \eqref{eq:straighten}, it is equivalent to setting $\psi=\pi/3$ and $\varphi = s - \pi/3$. Then $$2\sin\frac{\dj}{2}\leq 2\left(\frac{1}{2} + \sin\left(\frac{s}{2} - \frac{\pi}{6}\right)\right).$$ Therefore it is enough to verify that $$\frac{1}{2} + \sin\left(\frac{s}{2} - \frac{\pi}{6}\right)\leq\sin\left(s - \frac{\pi}{6}\right).$$
The last inequality holds because $$\frac{\partial^2}{\partial s^2}\left(\sin\left(s - \frac{\pi}{6}\right) - \sin\left(\frac{s}{2} - \frac{\pi}{6}\right)\right) = \frac{1}{4}\sin\left(\frac{s}{2} - \frac{\pi}{6}\right) - \sin\left(s - \frac{\pi}{6}\right) < 0$$
for $\pi/3\leq s\leq 2\pi/3$ and
$$\frac{1}{2} = \sin\left(s - \frac{\pi}{6}\right) - \sin\left(\frac{s}{2} - \frac{\pi}{6}\right)$$
for $s\in\{\pi/3, 2\pi/3\}$.

\caseitem[subsec:case-001]{1b}{$I = \{B_{i+1}\}$}

There are two direction jumps: between $B_i$ and $B_{i+1}$ of size $\pi/3$ and between $B_{i+1}$ and $B_j$.

Let $t$ be the point of the latter direction jump.
By Observation~\ref{lemma:inside-the-rhombus}, the point $t$ cannot lie inside the rhombus $cc_{i+1}c'c_{j-1}$, and therefore, the point $c'$ lies in the pentagon $c_{j}c_{j-1}cc_{i+1}t$.
Hence, if we apply Lemma~\ref{lemma:master} to the curve following the perimeter of $B_{i+1}$ from $t$ to $c'$ and then passing the perimeter of $B_{j-1}$ from $c'$ to $c_j$, we obtain
\begin{equation}
\angccw(t - c_{i+1}, c' - c_{i+1}) + \angccw(c' - c_{j-1}, c_j - c_{j-1})\geq\frac{\pi}{3}. \label{eqn:with-t}
\end{equation}

Since $B_{i+1}$ is involved in $\sigma_{ij}$, we have
\begin{align}
\angccw(t - c_{i+1}, c' - c_{i+1}) & = \angccw(c_i - c_{i+1}, c' - c_{i+1}) - \angccw(c_i - c_{i+1}, t - c_{i+1}) \notag \\ 
& \leq \angccw(c_i - c_{i+1}, c' - c_{i+1}) - \pi / 3 \label{eqn:rotate-equilateral} \\ 
& = \angle(v_i, u_j) - \pi / 3. \notag
\end{align}

\begin{figure}[h!]
    \centering
    \begin{subfigure}[t]{.4\textwidth}
    \includegraphics[width=\textwidth]{pics/huge-formula-pic.mps}
    \caption{Case~\ref{subsec:case-001}}
    \label{fig:huge-formula}
    \end{subfigure}
    \begin{subfigure}[t]{.4\textwidth}
    \includegraphics[width=\textwidth]{pics/case-3.mps}
    \caption{Case~\ref{subsec:case-101}}
    \label{fig:case-1c}
    \end{subfigure}
    \caption{}
\end{figure}

We are ready to bound $\Dj$. One may refer to Figure~\ref{fig:huge-formula} to follow the explanation.
\begin{align*}
\Dj & = \frac{\pi}{3} + \angccw(t - c_j, t - c_{i+1}) \\
& = \frac{\pi}{3} + \pi - \angccw(t - c_{i+1}, c_j - t)\\
& = \frac{4\pi}{3} - \angccw(t - c_{i+1}, c' - c_{i+1}) + \angccw(c' - c_{j-1}, c' - c_{i+1}) \\ & - \angccw(c' - c_{j-1}, c_j - c_{j-1}) - \angccw(c_j - c_{j-1}, c_j - t).
\end{align*}

Here, the second equality follows from replacing $\angccw(t - c_j, t - c_{i+1})$ by its adjacent angle. Then, since $\angccw(t - c_{i+1}, c_j - t)$ is, by definition, the angle we need to rotate the vector $t - c_{i+1}$ by in order to obtain the vector $c_j - t$; we may first rotate it counterclockwise until we obtain $c' - c_{i + 1}$, then clockwise until $c' - c_{j-1}$, then counterclockwise until $c_j - c_{j-1}$, and, finally, counterclockwise until $c_j - t$. This implies the last equality.

We continue bounding $\Dj$.
\begin{align*}
\Dj & \leq \pi + \psi - (\angccw(t - c_{i+1}, c' - c_{i+1}) + \angccw(c' - c_{j-1}, c_j - c_{j-1})) \\
& \stackrel{\eqref{eqn:with-t}}{\leq} \psi + 2(\angccw(t - c_{i+1}, c' - c_{i+1}) + \angle(u_i, v_j)) \\
& \stackrel{\eqref{eqn:rotate-equilateral}}{\leq} -\frac{2\pi}{3} + \psi + 2(\angle(v_i, u_j) + \angle(u_i, v_j)) \\
& = 3\psi - \frac{2\pi}{3} + 2\varphi.
\end{align*}

Here the first inequality is obtained after applying $\angccw(c' - c_{j-1}, c' - c_{i+1}) = \psi$ and $\angccw(c_j - c_{j-1}, c_j - t)\geq\pi/3$, which follows from the fact that $|t - c_{j-1}|\geq 1$. Due to Observation~\ref{lemma:psi-plus-phi}, the last equality holds.

\caseitem[subsec:case-100]{1b'}{$I = \{B_{j-1}\}$}

This case is similar to the previous one.

\caseitem[subsec:case-101]{1c}{$I = \{B_{i+1}, B_{j-1}\}$}

There are three direction jumps: between $B_i$ and $B_{i+1}$ of size $\pi/3$, between $B_{i+1}$ and $B_{j-1}$ of size $\psi$, and between $B_{j-1}$ and $B_j$ of size $\pi/3$.

According to Observation~\ref{lemma:psi-plus-phi}, $\psi + \varphi = \angle(u_i, v_j) + \angle(v_i, u_j) \ge \pi/3 + \pi/3$, see Figure~\ref{fig:case-1c}. Thus,
$$\Dj = \frac{2\pi}{3} + \psi \leq 3\psi - \frac{2\pi}{3} + 2\varphi,$$
which completes the proof of this case.


\caseitem[subsec:case-010]{1d}{$I = \{B\}$}

There are two direction jumps: between $B_i$ and $B$ and between $B$ and $B_j$.

By the definition, the direction jumps are exactly $\angccw(c_i - c_{i+1}, c - c_{i+1})$ and $\angccw(c - c_{j-1}, c_j - c_{j-1})$. Hence, Corollary~\ref{lemma:middle-involved-side-angles} finishes the proof.

\caseitem[subsec:case-011]{1e}{$I = \{B_{i+1}, B\}$}

There are three direction jumps: between $B_i$ and $B_{i+1}$ of size $\pi/3$, between $B_{i+1}$ and $B_{i+2}$ of size $\pi/3$, and between $B_{i+2}$ and $B_j$. According to Observation~\ref{lemma:not-covered-2pi-3}, we have $\angccw(c - c_{j-1}, c_j - c_{j-1})\geq2\pi/3$, and thus, the sum of first two direction jumps is no more than $\angccw(c_i - c_{i+1}, c_{i+2} - c_{i+1})$; hence,
$$\Dj\leq \angccw(c_i - c_{i+1}, c - c_{i+1}) + \angccw(c - c_{j-1}, c_j - c_{j-1}).$$
By Corollary~\ref{lemma:middle-involved-side-angles}, this does not exceed $3\psi - \frac{2\pi}{3} + 2\varphi$.


\caseitem[subsec:case-110]{1e'}{$I = \{B, B_{j-1}\}$}

This case is similar to the previous one.

\caseitem[subsec:case-111]{1f}{$I = \{B_{i+1}, B, B_{j-1}\}$}

There are four direction jumps, all of size $\pi/3$. 

By Observation~\ref{lemma:not-covered-2pi-3} we have $\psi\geq2\pi/3$, and by Observation~\ref{lemma:involved-source} we have $\varphi\geq 0$. Then

$$3\psi - \frac{2\pi}{3} + 2\varphi\geq 2\pi - \frac{2\pi}{3} = \frac{4\pi}{3} = \Dj.$$


\end{enumerate}

\section{Discussion}

We have shown that $f(3) = 37$, where $f(n)$ is the maximum possible number of disks in a packing of kissing radius $n$. In other words, triangular lattice provides the optimal size of the packing. It is not known whether $f(4) = f(3) + 24 = 61$ or not.

\bibliographystyle{unsrt}
\bibliography{main}

\newpage

\appendix
\begin{appendices}

\section{Proof of Claim~\ref{lemma:can-construct-gamma-ij}}\label{section:proof-of-claim}

Before proving Claim~\ref{lemma:can-construct-gamma-ij}, we first state some properties of Delaunay triangulations. After this we formulate several lemmas, which we then use to finish the proof. In this section, unless specified otherwise, by $\angle pqr$ we mean the smallest of two angles between the rays $[qp)$ and $[qr)$.

\subsection{Properties of Delaunay triangulations}

Recall that a Delaunay triangulation of a finite set of points not belonging to the same line is a triangulation where, for each triangle $pqr$ and any other point $s$ of the set, $s$ does not lie inside the circumcircle of $pqr$. If a Delaunay triangulation is fixed, we call any of its triangles just a \emph{Delaunay triangle}. In particular, if $pqr$ and $qps$ are two Delaunay triangles, then $\angle prq + \angle qsp\leq\pi$. For any triangle $\Delta$ denote by $R(\Delta)$ the circumradius of $\Delta$.

\begin{lemma}\label{lemma:no-obtuse-in-delaunay}
Let $pqrs$ be a convex quadrilateral, such that $pqr$ and $rsp$ are the faces of its Delaunay triangulation. If $R(pqr) < 1 \leq R(rsp)$, then $\angle rsp < \pi/2$.
\end{lemma}

\begin{proof}
We argue by contradiction. Suppose that $\angle rsp \geq \pi/2$. Then, since we have a Delaunay triangulation, $\angle pqr + \angle rsp \leq \pi$, or, equivalently, $\angle pqr \leq \pi - \angle rsp \leq \pi/2$. Thus, $\sin\angle pqr \leq \sin\angle rsp$.

Then, by the law of sines, $$2R(pqr) = \frac{|p - r|}{\sin\angle pqr} \geq \frac{|p - r|}{\sin\angle rsp} = 2R(rsp),$$ which contradicts the assumption.
\end{proof}

Note that the essential condition in this lemma is $R(pqr) < R(rsp)$; however, we strenghtened it by separating both sides with $1$ for more clear applications in the future.





\begin{lemma}\label{lemma:delaunay-sorted-process}
Given a set $S$ of $m$ points on the plane, not belonging to the same line, consider all non-degenerate triangles with vertices among these points in the nondecreasing order of the circumradius. Start with $\mathcal{F} = \varnothing$, and at each step, if the interior of the considered triangle does not contain any of the points and does not intersect anything from $\mathcal{F}$, then add the considered triangle to $\mathcal{F}$.

Then the set $\mathcal{F}$ at the end of the process is the set of faces of some Delaunay triangulation of $S$.
\end{lemma}

\begin{remark}A similar minimization property of Delaunay triangulations was already discovered, for example, in~\cite[Theorem~4]{Musin1997Delaunay}.
\end{remark}

\begin{proof}
It is obvious that at the end $\mathcal{F}$ will be a triangulation of $S$. Indeed, suppose that $\cup\mathcal{F}\neq\mathrm{conv}\,S$. This means that there is a point $p\in\mathrm{conv}\,S\setminus\cup\mathcal{F}$, and we may assume that $p$ does not belong to any segment between any two points of $S$, as otherwise the set $\mathrm{conv}\,S\setminus\cup\mathcal{F}$ would have zero area. Then there is at least one way to finish the triangulation, hence the set of all triangles containing $p$ and whose interiors do not intersect $\cup\mathcal{F}$ is not empty. Therefore, we should have added to $\mathcal{F}$ any of such triangles with minimum circumradius.

Also, for any triangle $pqr\in\mathcal{F}$, no point $t$ belongs to its boundary. Indeed, if there is a point $t\in [p, q]$, then we should have added $ptr$ or $qtr$ before $pqr$.

If the final triangulation is not Delaunay, then it is possible to do a ``flip''; that is, there is a convex quadrilateral $pqrs$ such that ${pqr}\in\mathcal{F}$ and ${rsp}\in\mathcal{F}$, and also $\angle{pqr} + \angle{rsp} > \pi$.

\begin{figure}[h!]
    \centering
    \includegraphics{pics/iterative-delaunay.mps}
    \caption{Lemma~\ref{lemma:delaunay-sorted-process}}
    \label{fig:delaunay-flip}
\end{figure}


Without loss of generality, assume that $R({pqr})\leq R({rsp})$. Let $\omega$ be the circumcircle of ${pqr}$. Since $\angle{pqr} + \angle{rsp} > \pi$, we know that $s$ lies inside $\omega$. This implies that $\angle{rsq} > \angle{rpq}$ and $\angle{qsp} > \angle{qrp}$. Since $\angle{rsq} + \angle{qsp} = \angle{rsp} < \pi$, at least one of $\angle{rsq}$ and $\angle{qsp}$ is acute. If $\angle{rsq} < \pi/2$ then $\sin\angle{rsq} > \sin\angle{rpq}$, and, by the law of sines,
$$R({rsq}) = R({rpq})\cdot\frac{\sin\angle{rpq}}{\sin\angle{rsq}} < R({rpq})\leq R({rsp}).$$
Then we should have considered ${rsq}$ before any of ${pqr}$ and ${rsp}$ and added it to~$\mathcal{F}$, which leads to a contradiction.

The case when $\angle{qsp} < \pi/2$ is analogous.
\end{proof}

We also remind that a \emph{Voronoi diagram} of a set $S$ of points is the division of the plane into (possibly unbounded) regions, which are also called \emph{cells}, and the region corresponding to a point $p\in S$ is defined as
$$\{x\in\mathbb{R}^2\,\colon\,\forall q\in S\quad |x - p|\leq |x - q|\}.$$
It is clear that each cell is a polyhedron, that is, an intersection of some halfplanes. It is also known that Voronoi diagram is dual to Delaunay triangulation in a sense that two points are connected by an edge in the Delaunay triangulation if and only if their corresponding Voronoi cells share a side, with some minor nuances.

In particular, a set of points may have several Delaunay triangulations, if there is a circle containing more than three points of $S$ on its boundary and no points from $S$ inside. In this case the points on this circle may be triangulated arbitrarily, and the Voronoi cells of any of them contains the center of this circle. So it would be formally correct to say that if two Voronoi cells share a side, then the corresponding points are connected in every Delaunay triangulation, and if two Voronoi cells share a point, then the corresponding points may or may not be connected in a Delaunay triangulation.

We will use the duality in the following form: If for some points $p$ and $q$ from $S$ and $x$ from $\mathbb{R}^2$ we have
$$\forall r\in S\setminus\{p, q\}\quad |x - p| = |x - q| < |x - r|,$$
then $p$ and $q$ are connected in any Delaunay triangulation of $S$.

\subsection{Auxiliary lemmas}

Recall that we are given a packing $\mathcal{P}$, and $(c_1, \ldots, c_n)$ is the cyclic sequence of all centers of disks from $\mathcal{P}$, in the order of traversal. The corresponding disks of unit diameter are denoted by $D_i$, and $B_i$ is the open disk with unit radius, centered at $c_i$. We say that $\mathcal{D}_{ij} = (D_i, \ldots, D_j)$ is a subsegment, if $D_i$ and $D_j$ are two consecutive $2$-disks. For any subsegment $\mathcal{D}_{ij}$ we denote by $R_{ij}$ the region bounded by segments $[c_i, c_{i+1}]$, \ldots, $[c_{j-1}, c_j]$, and two rays going from $c_i$ and $c_j$ in the direction from the origin, which is also the center of the $0$-disk; see Figure~\ref{fig:regions} on page~\pageref{fig:division-into-regions}. The union of all $B_i$ is denoted by $S$, and $\partial{S}\cap R_{ij}$ is denoted by $\sigma_{ij}$. If $D_k\in\mathcal{D}_{ij}$ and $\partial{B_k}\cap\sigma_{ij}\neq\varnothing$, we say that $B_k$ is \emph{involved} in $\sigma_{ij}$.

We want to show that it is possible to remove some points from the sequence $(c_1, \ldots, c_n)$, so that, if we construct the curve $\gamma$ on the remaining points, it will be valid, in particular, in terms of going counterclockwise.

Let $\mathcal{F}$ be a Delaunay triangulation of the set of points $\{c_1, \ldots, c_n\}$, and let $\{F_1, \ldots\}$ be the set of its triangular faces.
It can be seen that $\mathcal{F}$ contains an edge between points $c_i$ and $c_{i+1}$. Indeed, if we consider the Voronoi diagram of the set of disks $\{c_1, \ldots, c_n\}$, then, since disks $D_i$ and $D_{i+1}$ touch each other, their common point lies on the common edge of Voronoi cells corresponding to them, as this point is strictly outside all other disks. Since $\mathcal{F}$ is dual to the Voronoi diagram, points $c_i$ and $c_{i+1}$ are connected in $\mathcal{F}$.

Let $T = \bigcup_{i=1}^n[c_i, c_{i+1}]$ be the drawing of the $\mathcal{P}$-tree on the plane. We also define $$E = \bigcup\{F_i\,\colon\, R(F_i) < 1\}\cup T.$$

\begin{remark}The set of all Delaunay triangles and edges with \emph{covering radius} not exceeding $1$ (that is, triangles and edges that can be covered by a disk of unit radius) corresponds to a simplicial complex known as \emph{$\alpha$-complex}; see the survey~\cite{alpha_shapes}. We, however, are interested in circumradius instead of covering radius. It is also crucial that the circumradius of a triangle must be strictly less than $1$.
\end{remark}

Here and below, if $t$ is a point, we denote the open disk of unit radius centered at $t$ by $B_t$.

\begin{lemma}\label{lemma:small-R-is-good}
Let $pqr$ be a triangle. Let $K_q$ be the convex cone $\{q + \alpha(p - q) + \beta(r - q)\,\colon\,\alpha, \beta\geq 0\}$, and let $s_q$ be the arc $\partial B_q\cap K_q$. If $R(pqr) < 1$, then $s_q\subset B_p\cup B_r$. 
\end{lemma}

\begin{proof}
Let $C_q$ be the part of the Voronoi cell of points $\{p, q, r\}$ corresponding to the point $q$, restricted on $K_q$. Let $o$ be the circumcenter of $pqr$. It can be shown that $C_q$ lies inside the circle with diameter $oq$. Since $|o - q| = R(pqr) < 1$, it implies that this circle, in turn, belongs to $B_q$. Therefore, $s_q$ is covered by Voronoi cells of $p$ and $r$; thus, for every point $t$ of $s_q$ either $|t - p|$ or $|t - r|$ is less than $|t - q|$, which is $1$. Therefore, $s_q\subset B_p\cup B_r$.
%
\end{proof}

\begin{lemma}\label{lemma:E-is-connected}
$E$ is simply connected, or, equivalently, $\mathbb{R}^2\setminus E$ is connected.
\end{lemma}

\begin{proof}
Note that every triangle with circumradius less than $1$ has to be in a single region $R_{ij}$.
Indeed, suppose that some Delaunay triangle $pqr$ is not contained in any $R_{ij}$. This means that for some $2$-disk $D_k$ the triangle $pqr$ intersects the ray $[c_kf_k)$, where $f_k = c_k + \frac{c_k}{|c_k|}$.
Indeed, since, without loss of generality, the segment $[p, q]$ has to intersect the border between regions, and it cannot intersect $T$, we can assume that $[p, q]$ intersects some $[c_kf_k)$. We may additionally assume that neither $p$ nor $q$ coincides with $c_k$.

As shown in the proof of Observation~\ref{lemma:far}, angles $\angle pc_kc$ and $\angle cc_kq$ do not exceed $2\pi/3$, as each of them can be represented as the sum of two angles not exceeding $\pi/3$. Since the segment $[p, q]$ intersects the ray $[c_kf_k)$, we have $\angle pc_kq\geq2\pi/3$. This also implies that $|p - q|\geq \sqrt{3}$, due to the law of cosines.

If $r = c_k$, then $R(pc_kq)\geq 1$ due to the law of sines. Otherwise we can assume that $r$ and $c_k$ are at the different sides from the line through $p$ and $q$, as in the other case one we could take one of $[q, r]$ and $[p, r]$ instead of $[p, q]$. Since $pqr$ is a Delaunay triangle, the circumcircle of $pqr$ does not contain $c_k$. Therefore, $\angle prq\leq\pi - \angle pc_kq\leq\pi/3$, and due to the law of sines, $2R(pqr)\geq \frac{|p - q|}{\sin\angle prq}\geq 2$.

\medskip
Therefore, it suffices to prove the lemma separately for each subsegment $\mathcal{D}_{ij}$. More specifically, we need to prove that among all Delaunay faces within the corresponding region no triangle with small circumradius ``blocks'' a triangle with large circumradius. The case when $j = i + 2$ is trivial, so we will stick to the case when $j = i + 4$. For simplicity, we may assume that $i = 1$ and $j = 5$, and thus, $c_3 = c$.

We will keep in mind that $|c_k - c_{k+1}| = 1$ for every $k$. In particular, the inequality $R(c_{k-1}c_kc_{k+1}) < 1$ is equivalent to the inequality $\angle c_{k-1}c_kc_{k+1} < 2\pi/3$. We will also only consider triangles which are in $R_{ij}$.

Without loss of generality, we can assume that one of the following cases takes place; see Figure~\ref{fig:delaunay-blocking-cases} (in all cases below, when we mention $R(\Delta)$, we assume that $\Delta$ is a Delaunay triangle).

\begin{figure}[h!]
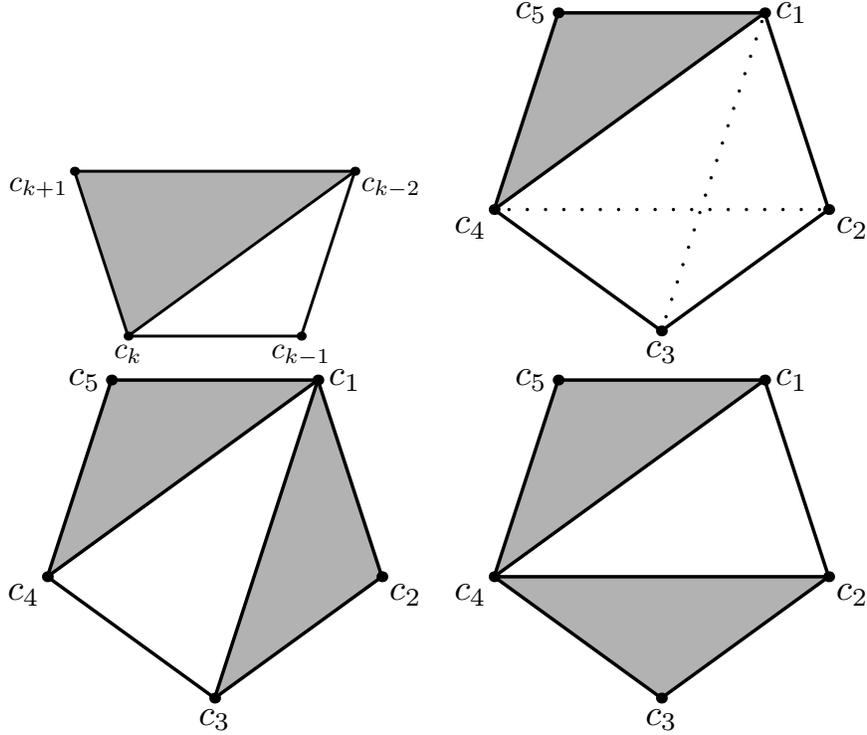

    \centering
    \begin{subfigure}[t]{.48\textwidth}
    \includegraphics[width=.95\textwidth]{pics/delaunay-blocking-case-1.mps}
    \end{subfigure}
    \begin{subfigure}[t]{.48\textwidth}
    \includegraphics[width=.95\textwidth]{pics/delaunay-blocking-case-2.mps}
    \end{subfigure}
    \begin{subfigure}[t]{.48\textwidth}
    \includegraphics[width=.95\textwidth]{pics/delaunay-blocking-case-3.mps}
    \end{subfigure}
    \begin{subfigure}[t]{.48\textwidth}
    \includegraphics[width=.95\textwidth]{pics/delaunay-blocking-case-4.mps}
    \end{subfigure}
    \caption{Cases from Lemma~\ref{lemma:E-is-connected}. Gray triangles have circumradii less than $1$.}
    \label{fig:delaunay-blocking-cases}
\end{figure}

\begin{itemize}
    \item $R(c_{k-2}c_{k-1}c_k)\geq 1 > R(c_{k-2}c_kc_{k+1})$ for $k = 3$, $k = 4$, or $k = 5$ (if we assume that $c_6 = c_1$).
    Since $R(c_{k-2}c_{k-1}c_k)\geq 1$, we have $\angle c_{k-2}c_{k-1}c_k\geq 2\pi/3$. Furthermore, since $R(c_{k-2}c_kc_{k+1}) < 1$, all angles of the triangle $c_{k-2}c_kc_{k+1}$ are less than $2\pi/3$, as otherwise it would have a side of length at least $\sqrt3$ opposed to the angle at least $2\pi/3$, which would contradict the law of sines. This implies that all inner angles of the quadrilateral $c_{k-2}c_{k-1}c_kc_{k+1}$ are less than $\pi$, or, equivalentlty, that the quadrilateral $c_{k-2}c_{k-1}c_kc_{k+1}$ is convex. This contradicts Lemma~\ref{lemma:no-obtuse-in-delaunay}. 
    
    \item $R(c_1c_4c_5) < 1$, $c_1c_2c_3c_4$ is a quadrilateral (maybe non-convex), which is divided by $\mathcal{F}$ into two triangles, both having circumradius at least $1$.
    In this case, either $\angle c_1c_2c_3$ in the quadrilateral $c_1c_2c_3c_4$ is at least $\pi$, and hence no more than $4\pi/3$, because $\angle c_1c_2c_4\leq\pi$ and $\angle c_4c_2c_3\leq\pi/3$; or $R(c_1c_2c_3)\geq 1$ due to Lemma~\ref{lemma:delaunay-sorted-process}, and $\angle c_1c_2c_3\geq 2\pi/3$. Similarly, $\angle c_2c_3c_4\in[2\pi/3, 4\pi/3]$.
    Then $|c_1 - c_4|\geq 2$, because the distance between the projections of $c_1$ and $c_4$ onto the line $(c_2, c_3)$ is at least $1 + 2\cos(\pi/3)$, which contradicts the fact that $R(c_1c_4c_5) < 1$.
    
    \item $R(c_1c_4c_5) < 1$, $R(c_1c_2c_3) < 1$, $R(c_1c_3c_4)\geq 1$. By the law of sines for $c_3c_1c_4$, we obtain
    $$\frac{|c_3 - c_4|}{\sin\angle c_3c_1c_4} \geq 2 \Rightarrow \sin\angle c_3c_1c_4\leq\frac12.$$
    Since the shortest side of $c_1c_3c_4$ is $c_3c_4$, this excludes the possibility that $\angle c_3c_1c_4\geq5\pi/6$, leaving us with $\angle c_3c_1c_4\leq\pi/6$.
    
    Then, as $R(c_1c_2c_3) < 1$, we know that $|c_1 - c_3| < \sqrt{3}$. By the law of sines for $c_1c_4c_3$, we have
    $$\frac{|c_1 - c_3|}{\sin\angle c_1c_4c_3} \geq 2 \Rightarrow \sin\angle c_1c_4c_3 < \frac{\sqrt 3}{2}.$$ Since $\angle c_1c_4c_3\leq \pi - \angle c_1c_2c_3\leq2\pi/3$, we have $\angle c_1c_4c_3 < \pi/3$. In particular, we know now that all of the angles $\angle c_3c_4c_1$, $\angle c_1c_4c_5$, $\angle c_5c_1c_4$, $c_4c_1c_3$ are acute, which means that the quadrilateral $c_1c_5c_4c_3$ is convex. Finally, by Lemma~\ref{lemma:no-obtuse-in-delaunay}, $\angle c_1c_3c_4 < \pi/2$, which leads to $c_1c_3c_4$ having angles with sum less than $\pi$, thus a contradiction.
    
    \item $R(c_1c_4c_5) < 1$, $R(c_2c_3c_4) < 1$, $R(c_1c_2c_4)\geq 1$. This case is completely analogous to the previous one.
\end{itemize}
\end{proof}

According to Lemma~\ref{lemma:E-is-connected}, all triangles $\Delta$ of $\mathcal{F}$ with $R(\Delta) < 1$ can be ordered in such a way, that, if we add them to $T$ one by one, the union is always simply connected. Formally, if $\mathcal{F} = \{F_1, \ldots, F_m\}$, define $\mathcal{F}_k = \{F_1, \ldots, F_k\}$. Also define $E_0 = T$ and $E_k = E_{k-1}\cup F_k$ for all $k\in[m]$. In particular, $\mathcal{F}_0 = \varnothing$, $\mathcal{F}_m = \mathcal{F}$, and $E_m = E$. Then, due to Lemma~\ref{lemma:E-is-connected}, we may assume that all $E_i$ are simply connected. 
In particular, in the rest of the paper, when we write $F_i = pqr$, we assume that $p$, $q$, $r$ are three consecutive vertices of $\partial E_i$ in this order, unless specified otherwise.

Let $(d_1, \ldots, d_k) = (c_{i_1}, \ldots, c_{i_k})$ be the cyclic sequence of all endpoints of the counterclockwise traversal of $\partial{E}$, and let $I = \{i_1, \ldots, i_k\}$.
If $p$, $q$, $r$ are some points with $q\neq p, r$, then define $\ccwarc{q}{p}{r}$ as the arc of $\partial B_q$ going from the ray $[qp)$ to $[qr)$ counterclockwise. If the rays $[qp)$ and $[qr)$ coincide, we assume that $\ccwarc{q}{p}{r}$ has length $2\pi$, not $0$.
For all $j$ denote by $s_j$ the arc $\ccwarc{d_j}{d_{j-1}}{d_{j+1}}$.
For all $j$ put
$$K_j = d_j + \{\lambda (x - d_j)\,\colon\,\lambda\geq 0, x\in s_j\}.$$
In particular, $s_j = \partial B_{i_j}\cap K_j$. Note that, unlike in Lemma~\ref{lemma:small-R-is-good}, the cone $K_j$ does not have to be convex.

The following is the inverse of Lemma~\ref{lemma:small-R-is-good}.

\begin{lemma}\label{lemma:covered-arc-implies-small-R}
$s_j\not\subset B_{i_{j-1}}\cup B_{i_{j+1}}$.
\end{lemma}

\begin{proof}
For simplicity, let $p = d_{j-1}$, $q = d_j$, $r = d_{j+1}$. Also, define for consistency $s_q = s_j$ and $K_q = K_j$.

First of all, if $|s_j|\geq\pi$, then both $B_{i_{j-1}}$ and $B_{i_{j+1}}$ intersect $s_j$ over an arc of length at most $\pi/3$, which implies the required relation. Therefore, we may assume that $|s_j| < \pi$. In particular, it means that $p\neq r$.

Suppose that $\angle qrp > \pi/2$. This implies that the side $pq$ is the largest in the triangle $pqr$, thus is longer than $1$. Hence, there is the triangle $F_i\in\mathcal{F}$ such that $F_i = psq$ for some $s$.

Since $s$ and $r$ are at different sides from $pq$, we have $\angle psq + \angle qrp\leq\pi$, because $psq$ and $pqr$ are Delaunay triangles. But then $\sin\angle psq\leq\sin(\pi - \angle qrp) = \sin\angle qrp$, which, combined with the law of sines, contradicts the inequality $R(psq) < R(pqr)$. Therefore, the case $\angle qrp > \pi/2$ is impossible; similar to the case $\angle rpq > \pi/2$.

Since none of $\angle qrp$ and $\angle rpq$ is obtuse,
the proof of Lemma~\ref{lemma:small-R-is-good} works in the opposite way.
Indeed, let $o$ be the circumcenter of $pqr$. We know that $o$ is inside $K_q$; thus, the segment $[o, q]$ intersects $s_q$. If $t$ is the intersection point, then $|t - p|\geq |t - q| = 1$ and $|t - r|\geq |t - q| = 1$; thus, $t\in s_q\setminus(B_p\cup B_r)$.
\end{proof}

\subsection{Finishing the proof of Claim~\ref{lemma:can-construct-gamma-ij}}

The following two lemmas are the main lemmas of the whole proof. Lemma~\ref{lemma:delaunay-covers-neighbors} states that $I$ is the sequence of all involved disks, and hence, we only need those disks to build a curve containing $\partial S$. The last lemma establishes the fact that the curve built as formulated in Claim~\ref{lemma:can-construct-gamma-ij} indeed covers $\partial S$.

\begin{lemma}\label{lemma:delaunay-covers-neighbors}
If $\mathcal{D}_{ij}$ is a subsegment and $D_k\in\mathcal{D}_{ij}$, then $k\in I$ if and only if $B_k$ is involved in $\sigma_{ij}$. In particular, if $D_i$ is a $2$-disk, then $i\in I$.
\end{lemma}

\begin{proof}
The proof is divided into two parts. First we prove that if $B_k$ is involved in some $\sigma_{ij}$, then $k\in I$. After it we prove the converse.

Suppose that $k\notin I$. Then the arc $\ccwarc{c_k}{c_{k-1}}{c_{k+1}}$ can be divided into smaller arcs by all edges of $\mathcal{F}$ that have $c_k$ as one of the endpoints. More specifically, $\ccwarc{c_k}{c_{k-1}}{c_{k+1}}$ is divided into arcs of type $\ccwarc{c_k}{p}{q}$, where $c_kpq\in\mathcal{F}$. But, according to Lemma~\ref{lemma:small-R-is-good}, each arc $\ccwarc{c_k}{p}{q}$ is covered by $B_p$ and $B_q$. This implies that $B_k$ is not involved in any $\sigma_{ij}$.

\medskip
If $D_i$ is a $2$-disk, then, due to Observation~\ref{lemma:far}, the point $f_i$ is not covered by any disk $B_k$. Therefore, what remains to prove is that if $\mathcal{D}_{ij}$ is a subsegment, then for any
$k$ such that $i < i_k < j$, the disk $B_{i_k}$ is involved in $\sigma_{ij}$.

We remind that $s_k = \ccwarc{d_k}{d_{k-1}}{d_{k+1}}$ by definition. If $p\in s_k$, we call the arc $\ccwarc{d_k}{d_{k-1}}{p}$ a \emph{prefix} of $s_k$, and the arc $\ccwarc{d_k}{p}{d_{k+1}}$ a \emph{suffix} of $s_k$.\footnote{Formally, if $p$ belongs to the ray $[d_kd_{k-1})$, then by the corresponding prefix we mean the point $p$ alone, while our definition specifies it to be the whole circle. Similarly, if $p\in[d_kd_{k+1})$, we say that $p$ is the suffix, not the arc of length $2\pi$.}

We know that $B_{i_{k-1}}$ and $B_{i_{k+1}}$ intersect $s_k$ by some prefix and some suffix of this arc.
Due to Lemma~\ref{lemma:covered-arc-implies-small-R}, we have $s_k\not\subset B_{i_{k-1}}\cup B_{i_{k+1}}$.
To prove that $B_{i_k}$ is involved in $\sigma_{ij}$, it suffices to show that there is a point at $s_k\setminus(B_{i_{k-1}}\cup B_{i_{k+1}})$ that is not covered by $S$. In fact, we will prove that no point of $s_k\setminus(B_{i_{k-1}}\cup B_{i_{k+1}})$ belongs to $S$. Suppose that some disk $B_l$ intersects the arc $s_k\setminus(B_{i_{k-1}}\cup B_{i_{k+1}})$. Due to Lemma~\ref{lemma:only_inner}, $D_l\in\mathcal{D}_{ij}$.
Without loss of generality we may assume that $i = 1$, $j = 5$, and $l < i_k$. Consider all possible cases, see Figure~\ref{fig:delaunay-covers-neighbors-cases}.

\begin{figure}[h!]
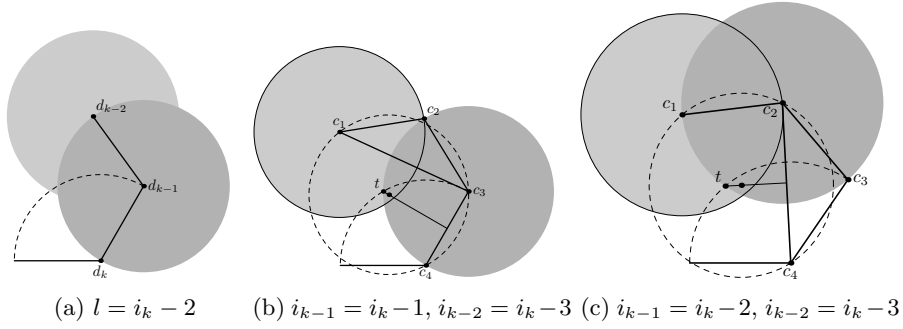

    \centering
    \begin{subfigure}[t]{.26\textwidth}
    \includegraphics[width=.95\textwidth]{pics/delaunay-covers-neighbors-case-1.mps}
    \caption{$l = i_k - 2$}
    \label{fig:delaunay-covers-neighbors-case-a}
    \end{subfigure}
    \begin{subfigure}[t]{.35\textwidth}
    \includegraphics[width=.95\textwidth]{pics/delaunay-covers-neighbors-case-2.mps}
    \caption{$i_{k-1} = i_k - 1$, $i_{k-2} = i_k - 3$}
    \label{fig:delaunay-covers-neighbors-case-b}
    \end{subfigure}
    \begin{subfigure}[t]{.35\textwidth}
    \includegraphics[width=.95\textwidth]{pics/delaunay-covers-neighbors-case-3.mps}
    \caption{$i_{k-1} = i_k - 2$, $i_{k-2} = i_k - 3$}
    \label{fig:delaunay-covers-neighbors-case-c}
    \end{subfigure}
    \caption{Cases from Lemma~\ref{lemma:delaunay-covers-neighbors}.}
    \label{fig:delaunay-covers-neighbors-cases}
\end{figure}

\begin{itemize}
    \item $i_{k-1} = i_k - 1$, $l = i_{k-2} = i_k - 2$. See Figure~\ref{fig:delaunay-covers-neighbors-case-a}. Since $B_{i_{k-1}}$ covers an arc of length $\pi/3$ of $s_k$, and $B_l$ covers the prefix of length $\pi - \angle{d_{k-2}d_{k-1}d_k}$, we have $\angle{d_{k-2}d_{k-1}d_k} < 2\pi/3$, which contradicts the fact that $R(d_{k-2}d_{k-1}d_k)\geq 1$; otherwise our process of adding triangles in ascending order would add this triangle as well. In particular, in this case $\pi - \angle{d_{k-2}d_{k-1}d_k} > 0$ must hold.
    \item $i_{k-1} = i_k - 1$, $l = i_{k-2} = i_k - 3$. See Figure~\ref{fig:delaunay-covers-neighbors-case-b}. In all remaining cases, including this, $l = i_k - 3$, which implies that $l = 1$ and $i_k = 4$. By the law of sines for $c_1c_4c_3$,
    $$\frac{|c_1 - c_3|}{\sin\angle c_1c_4c_3}\geq 2,$$
    hence, since $c_1c_2c_3\in\mathcal{F}$ and, consequently, $|c_1 - c_3|\leq\sqrt3$, we have
    $$\sin\angle c_1c_4c_3 < \frac{\sqrt{3}}{2},$$
    therefore,
    $$\angle c_1c_4c_3 < \frac{\pi}{3}\text{ or }\angle c_1c_4c_3 > \frac{2\pi}{3}.$$
    If $\angle{c_1c_4c_3} > 2\pi/3$, then $|c_1 - c_3| > \sqrt{3}$; therefore, $\angle c_1c_4c_3 < \pi/3$.
    
    Let $t\in s_k$ be the point such that $B_3$ covers the arc of $s_k$ between $c_3$ and $t$. In particular, the triangle $c_3c_4t$ is regular. Using our assumption that
    $B_1\cap(s_k\setminus\ccwarc{d_k}{c_3}{t})\neq\varnothing$,
    and the fact that $\angle c_3c_4c_1 < \pi/3$, we conclude that $|c_1 - t| < 1$ must hold. But then some point of the segment $[t, (c_4 + c_3)/2]$ must be the circumcenter of $c_1c_3c_4$. Since it is not $t$, we have $R(c_1c_3c_4) < 1$, which leads to a contradiction.
    \item $i_{k-1} = i_k - 2$, $l = i_{k-2} = i_k - 3$. See Figure~\ref{fig:delaunay-covers-neighbors-case-c}. By the law of sines for $c_1c_4c_2$, we have
    $$\frac{|c_1 - c_2|}{\sin\angle c_1c_4c_2}\geq 2,$$
    hence,
    $$\sin\angle c_1c_4c_2 \leq \frac{1}{2},$$
    therefore,
    $$\angle c_1c_4c_2 \leq \frac{\pi}{6}\text{ or }\angle c_1c_4c_2 \geq \frac{5\pi}{6}.$$
    Since $c_1c_2$ is the minimum side of $c_1c_2c_4$, we have the inequality $\angle c_1c_4c_2\leq\pi/3$; therefore, $\angle c_1c_4c_2\leq\pi/6$. Since $c_2c_3c_4\in\mathcal{F}$, we know that
    $$\angle c_3c_4c_2 > \pi/6\Rightarrow \angle c_1c_4c_2 < \angle c_3c_4c_2.$$
    
    Let $t\in s_k$ be the point such that $B_2$ covers the arc of $s_k$ between $c_3$ and $t$. In particular, $c_2c_3c_4t$ is a rhombus. Using our assumption that
    $B_1\cap(s_k\setminus\ccwarc{d_k}{c_3}{t})\neq\varnothing$,
    and the fact that $\angle c_2c_4c_1 < \angle c_2c_4c_3 = \angle c_2c_4t$, we conclude that $|c_1 - t| < 1$ must hold.
    
    Analogously to the previous case, $\angle{c_2c_1c_4} < \pi/3$. Therefore, $\angle c_2c_1c_4 < \pi/2$, which is equivalent to the inequality
    $$|c_1 - (c_2 + c_4)/2| > |(c_2 - c_4)/2|.$$
    But then, due to continuity, some point of the segment $[t, (c_2 + c_4)/2]$ must be the circumcenter of $c_1c_2c_4$. Since it is not $t$, we have $R(c_1c_2c_4) < 1$, which leads to a contradiction.
    \item $i_{k-1} = i_k - 1$, $l = i_{k-3} = i_k - 3$. Since $R(c_2c_3c_4)\geq 1$ and $R(c_1c_2c_3)\geq 1$, both $\angle c_2c_3c_4$ and $\angle c_1c_2c_3$ are at least $2\pi/3$. Therefore, $|c_1 - c_4|\geq 2$, which contradicts that $s_k\cap B_1\neq\varnothing$.
\end{itemize}

As we got a contradiction in each case, we conclude that $s_k$ is covered by $B_{i_{k-1}}$ and $B_{i_{k+1}}$.
\end{proof}


\begin{lemma}
If $\mathcal{D}_{ij}$ is a subsegment, and $p\in \sigma_{ij}\cap\partial{B_{i_k}}$ for some $D_{i_k}\in\mathcal{D}_{ij}$, then $p\in s_k$ and $p$ is not covered by any of $B_{i_{k-1}}$ and $B_{i_{k+1}}$.
\end{lemma}

\begin{proof}
By construction, $p\notin E$ because every point from $E$ is at distance less than $1$ from one of the points. This fact and Lemma~\ref{lemma:E-is-connected} immediately imply that $p\in s_k$.

Similarly, since $|p - c_{i_{k-1}}|\geq 1$ and $|p - c_{i_{k+1}}|\geq 1$, we also have that $p$ is not covered by any of $B_{i_{k-1}}$ and $B_{i_{k+1}}$.
\end{proof}

This finishes the proof of Claim~\ref{lemma:can-construct-gamma-ij}.



\end{appendices}

\end{document}